\documentclass[11pt,oneside,english]{amsart}
\usepackage[latin9]{inputenc}
\usepackage[colorlinks=true, linkcolor=blue, citecolor=black,
urlcolor=black]{hyperref}
\usepackage{geometry}
\usepackage{color}
\usepackage{babel}
\usepackage{algorithm2e}
\usepackage{amsthm}
\usepackage{amssymb}
\usepackage{cancel}
\usepackage{stackrel}
\usepackage[normalem]{ulem}

\makeatletter
\numberwithin{equation}{section}
\numberwithin{figure}{section}
\numberwithin{table}{section}
\theoremstyle{plain}
\newtheorem{thm}{\protect\theoremname}[section]
\theoremstyle{definition}

\theoremstyle{remark}
\newtheorem{rem}[thm]{\protect\remarkname}
\theoremstyle{plain}
\newtheorem{cor}[thm]{\protect\corollaryname}
\theoremstyle{plain}
\newtheorem{lem}[thm]{\protect\lemmaname}
\theoremstyle{plain}
\newtheorem{prop}[thm]{\protect\propname}

\theoremstyle{theorem}
\newtheorem{thmx}{Theorem}

\newcommand{\cal}{\mathcal}

\makeatother

\providecommand{\corollaryname}{Corollary}
\providecommand{\definitionname}{Definition}
\providecommand{\lemmaname}{Lemma}
\providecommand{\remarkname}{Remark}
\providecommand{\theoremname}{Theorem}
\providecommand{\propname}{Proposition}

\begin{document}
\title[Counting and Equidistribution]{Counting, equidistribution and entropy gaps at infinity with applications to
cusped Hitchin representations}
\author[Bray]{Harrison Bray}
\address{George Mason University, Fairfax, Virginia 22030}
\author[Canary]{Richard Canary}
\address{University of Michigan, Ann Arbor, MI 41809}
\author[Kao]{Lien-Yung Kao}
\address{George Washington University, Washington, D.C. 20052}
\author[Martone]{Giuseppe Martone}
\address{University of Michigan, Ann Arbor, MI 41809}
\thanks{
Canary was partially supported by grant  DMS-1906441 from the National Science Foundation
and grant 674990 from the Simons Foundation}

\begin{abstract}
We show that if an eventually positive, non-arithmetic, locally H\"older continuous potential  for a topologically mixing
countable Markov shift with (BIP) has an entropy gap at infinity,
then one may apply the renewal theorem of Kesseb\"ohmer and Kombrink to obtain counting and equidistribution
results. We apply these general results to obtain counting and equidistribution results for cusped Hitchin
representations, and more generally for cusped Anosov representations of geometrically finite Fuchsian groups.

\end{abstract}

\maketitle

\section{Introduction}

In this paper, we use the Renewal Theorem of Kesseb\"ohmer and Kombrink \cite{kess-kom}  to establish counting and equidistribution 
results for well-behaved  potentials on topologically mixing countable Markov shifts with (BIP) in the spirit of Lalley's work \cite{lalley}
on finite Markov shifts. Inspired by work of Schapira-Tapie \cite{schapira-tapie,schapira-tapie-counting},
Dal'bo-Otal-Peign\'e \cite{DOP}, Iommi-Riquelme-Velozo \cite{IRV} and Velozo \cite{velozo} in the setting of geodesic flows on 
negatively curved Riemannian manifolds,
we define notions of entropy gap at infinity for our
potentials. Our results require that the potentials are non-arithmetic, eventually positive and have an entropy gap at infinity.

Our main motivation for this general analysis was provided by cusped Hitchin  representations
of a geometrically finite Fuchsian group into $\mathsf{SL}(d,\mathbb R)$. Given a linear functional
$\phi$ on the Cartan algebra $\mathfrak{a}$ of $\mathsf{SL}(d,\mathbb R)$ which is a positive linear combination of
simple roots, we can define the 
$\phi$-translation length $\ell^\phi(A)=\phi(\ell(A))$ (where
 $\ell$ is the Jordan projection) for $A\in\mathsf{SL}(d,\mathbb R)$. The first consequence of the general theory we develop is that
if $\rho$ is cusped Hitchin, then
$$\#\big\{[\gamma]\in [\Gamma]\ |\ 0<\ell^\phi(\rho(\gamma))\le t\big\}\sim\frac{e^{t\delta}}{t\delta}$$
where $\delta=\delta_\phi(\rho)$ is the $\phi$-entropy of $\rho$
(and $[\Gamma]$ is the collection of conjugacy classes of elements of $\Gamma$.) We also
obtain a Manhattan curve theorem and equidistribution results in this context. In later work, we plan to use these results
to construct pressure metrics on cusped Hitchin components. A longer term goal is the development of a geometric theory
of the augmented Hitchin component which parallels the study of the augmented Teichm\"uller space as the
metric completion of Teichm\"uller space with the Weil-Petersson metric (see Masur \cite{masur-wp}).

\medskip

{\bf General Thermodynamical results:}
We now give more precise statements of our general results.
We assume throughout that $(\Sigma^+,\sigma)$ is a topologically mixing, one-sided, countable Markov shift
with alphabet $\mathcal A$
which has the big images and pre-images property (BIP). Moreover, all of our functions will be assumed to be locally H\"older continuous (see 
Section \ref{countable background} for precise definitions).

We now introduce the crucial assumptions we will make in our work.
Given a locally H\"older continuous function $f:\Sigma^+\to\mathbb R$ and $a\in\mathcal A$, we let
$$I(f,a)=\inf\big\{ f(x)\ |\ x\in\Sigma^+, x_1=a\big\}\qquad\mathrm{and}\qquad S(f,a)=\sup\big\{ f(x)\ |\ x\in\Sigma^+, x_1=a\big\}.$$
Note that $I(f,a)$ and $S(f,a)$ are finite since $f$ is locally H\"older continuous.

We say that $f$ has a {\em strong entropy gap at infinity}  if the series
$$Z_1(f,s)=\sum_{a\in\mathcal A} e^{-sS(f,a)}$$
has a finite critical exponent $d(f)>0$ and diverges when $s=d(f)$. 

We say that $f$ has a {\em weak entropy gap at infinity} if $Z_1(f,s)$ has a finite critical exponent $d(f)>0$ and there exists
$\delta=\delta(f)>d(f) >0$ so that  $P(-\delta f)=0$
where $P$ is the Gurevich  pressure function associated to $(\Sigma^+,\sigma)$ (defined in Section \ref{countable background}).
We will see later (in Section \ref{entropy gaps}), that a strong entropy gap at infinity implies a weak entropy
gap at infinity. 

We say that $f$ is {\em strictly positive} if $c(f)=\inf\{f(x)\ |\ x\in\Sigma^+\}>0$.
We say that $f$ is {\em eventually positive} if there exist $N\in\mathbb N$ and $B>0$ so that 
$$S_nf(x)=f(x)+f(\sigma(x))+\cdots+f(\sigma^{n-1}(x))>B$$ 
for all $n\ge N$ and $x\in\Sigma^+$. 
Recall that $f$ is {\em arithmetic} if the subgroup of $\mathbb R$ generated by
$\{ S_nf(x)\ | \ x\in \text{Fix}^n,\ n\in\mathbb N\}$ is cyclic, where $x\in\text{Fix}^n$ if $\sigma^n(x)=x$.

We begin by stating our general counting results.
For all $n\in\mathbb N$, let
$$\mathcal M_f(n,t)=\{x\in\Sigma^{+}:\ x\in\mathrm{Fix}^{n}\ \mathrm{and\ }S_{n} f(x)\leq t\}\  \ \mathrm{and\ let}\ \ 
M_f(t)=\sum_{n=1}^\infty \frac{1}{n} \#\mathcal M_f(n,t).$$

\begin{thmx}[Growth rate of closed orbits]\label{thm:countingN}
Suppose that $(\Sigma^+,\sigma)$ is a topologically mixing, one-sided, countable Markov shift
which has (BIP). If $f:\Sigma^+\to\mathbb R$ is locally H\"older continuous, non-arithmetic, eventually positive and has a weak entropy gap at infinity,  and
$P(-\delta f)=0$, then 
$${\displaystyle \lim_{t\to\infty}M_f(t)\frac{t\delta}{e^{t\delta}}}=1.$$
\end{thmx}

Similarly, for all $k\in\mathbb N$, let
$$\mathcal R_f(k,t)=\{ x\in\mathcal M_f(k,t)\  | \ x\notin \mathcal M_f(n,t)\ \mathrm{if}\ n<k\}\ \ \mathrm{and\ let}\ \ 
R_f(t)=\sum_{k=1}^\infty\frac{1}{k} \#\mathcal R_f(k,t).$$

If $x\in\mathcal M_f(n,t)-\mathcal R_f(n,t)$, then there exists $j\ge 2$ so that $x\in \mathcal M_f(\frac{n}{j},\frac{t}{j})$, so
$$M_f(t)-M_f\left(\frac{t}{2}\right)\leq R_f(t)\leq M_f(t).$$
Therefore, the following result is an immediate corollary of Theorem \ref{thm:countingN}.

\begin{cor}[Growth rate of closed prime orbits] \label{cor:counting closed}
Suppose that $(\Sigma^+,\sigma)$ is a topologically mixing, one-sided, countable Markov shift which has (BIP).
If $f:\Sigma^+\to\mathbb R$ is locally H\"older continuous, non-arithmetic, eventually positive and has a weak entropy gap at infinity, 
and $P(-\delta f)=0$, then 
\[
\lim_{t\to\infty}R_f(t)\frac{t\delta}{e^{t\delta}}=1.
\]
\end{cor}

If $f$ is strictly positive, let $\Sigma_f$ be the suspension flow of $f$. In this setting, we obtained a generalized form  
of Bowen's formula for the critical exponent. Let $\mathcal O_f$ be the collection of closed orbits of $\Sigma_f$ and let
$$\mathcal O_f(t)=\{\lambda \ |\ \ell_f(\lambda)\le t\}$$
where $\ell_f(\lambda)$ is the period of $\lambda$.
Notice that $\#\mathcal O_f(t)=M_f(t)$, since if $\lambda\in\mathcal O_f(t)$, then there exists $x\in\text{Fix}^n$ for some $n$, so that
$S_nf(x)=\ell_f(\lambda)$ and $x$ is well-defined up to cyclic permutation.
Lemma \ref{eventually to strictly} implies that every eventually positive locally H\"older continuous function (in our setting) is cohomologous to
a strictly positive locally H\"older continuous function, so we are always free to interpret our results from this viewpoint.

\begin{cor}
[Bowen's formula]
Suppose that $(\Sigma^+,\sigma)$ is a topologically mixing, one-sided, countable Markov shift
which has (BIP).
If $f:\Sigma^+\to\mathbb R$ is locally H\"older continuous, non-arithmetic, strictly positive, has a weak entropy gap at infinity
and $P(-\delta f)=0$, then 
\[
\delta=\lim_{t\to\infty}\frac{1}{t}\log\#\mathcal O_f(t).
\]
\end{cor}

If $f:\Sigma^+\to\mathbb R$ and $g:\Sigma^+\to\mathbb R$ are two strictly 
positive locally H\"older continuous functions, then
there is a natural identification of the set $\mathcal O_f$ of closed orbits of $\Sigma_f$ and
the set $\mathcal O_g$ of closed orbits of $\Sigma_g$.  If $f$ is strictly positive and has a weak entropy gap at infinity so that $P(-\delta f)=0$,
then the equilibrium state for $-\delta f$ induces a measure of maximal entropy on the suspension flow on $\Sigma_f$. We obtain an
equidistribution result for this equilibrium state which roughly says that it behaves like a  Patterson-Sullivan measure.

In the following theorem, if $\phi$ and $\psi$ are real-valued functions, we say that 
$$\phi\sim\psi\qquad\mathrm{if}\qquad \lim_{t\to\infty}\frac{\phi(t)}{\psi(t)}=1.$$

\begin{thmx}[Equidistribution] \label{thm:equid_roof}
Suppose that $(\Sigma^+,\sigma)$ is a topologically mixing, one-sided, countable Markov shift which has (BIP) and
$f:\Sigma^+\to\mathbb R$ is locally H\"older continuous, non-arithmetic, eventually positive, has a weak entropy gap at infinity, $P(-\delta f)=0$ and
$\mu_{-\delta f}$ is the equilibrium state for $-\delta f$.
If
$g:\Sigma^+\to\mathbb R$ is locally H\"older continuous, eventually positive, and there exists $C>0$
such that 
$$|f(x)-g(x)|<C$$ 
for all $x\in\Sigma^+$, then
\[
\sum_{k=1}^{\infty}\frac{1}{k}\sum_{x\in \mathcal M_f(k,t)}\frac{S_k g(x)}{S_k f(x)}\sim\left(\frac{\int g \ d\mu_{-\delta f}}{\int f\ d\mu_{-\delta f}}\right)\cdot\frac{e^{t\delta}}{t\delta}
\]
as $t\to\infty$. If $f$ and $g$ are strictly positive, then 
 \[
\sum_{\gamma\in\mathcal O_f(t)}\frac{l_{g}(\gamma)}{l_{f}(\gamma)}\sim\left(\frac{\int g\ d\mu_{-\delta f}}{\int f\ d\mu_{-\delta f}}\right)\cdot\frac{e^{t\delta}}{t\delta}
\]
as $t\to\infty.$
\end{thmx}

We can obtain a completely analogous statement if we instead consider the set $\mathcal P_f$ of primitive closed orbits of the suspension flow $\Sigma_f$.

Suppose that $f:\Sigma^+\to\mathbb R$ is locally H\"older continuous, eventually positive,  and 
has a strong entropy gap at infinity and that $g:\Sigma^+\to\mathbb R$ is also
eventually positive and  locally
H\"older continuous,  and that there exists $C>0$ so that $|f(x)-g(x)|<C$ for all $x\in\Sigma^+$.
(Notice that this implies that $d(f)=d(g)$.)
Inspired by Burger \cite{burger},
we define, the {\em Manhattan curve}
$$\mathcal C(f,g)=\{ (a,b)\in \mathbb R^2\ | \ P(-af-bg)=0\ \ a\ge 0, \ b\ge 0,\  a+b> 0\}.$$
The Manhattan curve has the following properties.

\begin{thmx}[Manhattan curve]
\label{Manhattan curve}
Suppose that $(\Sigma^+,\sigma)$ is a topologically mixing, one-sided countable Markov
shift with (BIP),
$f:\Sigma^+\to\mathbb R$ is locally H\"older continuous, eventually positive and 
has a strong entropy gap at infinity and  that $g:\Sigma^+\to\mathbb R$ is also
eventually positive and  locally H\"older continuous.
If there exists $C>0$ so that $|f(x)-g(x)|<C$ for all $x\in\Sigma^+$,
then 
\begin{enumerate}
\item
$(\delta(f),0),\ (0,\delta(g))\in \mathcal C(f,g)$.
\item 
If $a\ge 0$, $b\ge 0$, and $a+b> 0$, then there exists a unique $t>\frac{d(f)}{a+b}$ so
that $(ta,tb)\in\mathcal C(f,g)$.
\item
$\mathcal C(f,g)$ is a closed subsegment of an analytic curve.
\item
$\mathcal C(f,g)$ is strictly convex, unless
$$S_nf(x)=\frac{\delta(g)}{\delta(f)} S_n g(x)$$ 
for all $x\in\mathrm{Fix}^n$ and $n\in\mathbb N$.
\end{enumerate}
Moreover,  the tangent line to $\mathcal C(f,g)$ at $(a,b)\in\mathcal C(f,g)$ has slope
$$s(a,b)=-\frac{\int_{\Sigma^+}g\ d\mu_{-af-bg}}{\int_{\Sigma^+} f\ d\mu_{-af-bg}}$$
where $\mu_{-af-bg}$ is the equilibrium state of the function $-af-bg$.
\end{thmx}

\medskip\noindent
{\bf Applications to cusped Hitchin representations:} Let $S=\mathbb H^2/\Gamma$ be a geometrically finite, hyperbolic
surface, and let  $\Lambda(\Gamma)\subset \partial\mathbb H^2$ be the limit set of $\Gamma\subset\mathsf{PSL}(2,\mathbb R)$.
Following Fock and Goncharov \cite{fock-goncharov}, a {\em cusped Hitchin representation} is a representation 
$\rho:\Gamma\to \mathsf{SL}(d,\mathbb R)$
such that if $\gamma\in\Gamma$ is parabolic, then $\rho(\gamma)$ is a unipotent element with a single Jordan block
and there exists a $\rho$-equivariant positive map $\xi_\rho:\Lambda(\Gamma)\to\mathcal F_d$.
If $S$  is compact, cusped Hitchin representations are just the traditional Hitchin representations introduced by Hitchin \cite{hitchin} and further studied
by Labourie \cite{labourie-invent}, while if 
$\Gamma$ is convex cocompact,  they are the
Hitchin representations studied by Labourie-McShane \cite{labourie-mcshane}.
As these are covered by the traditional theory of Anosov representations, we will focus on the case where $\Gamma$ is not convex cocompact.
If $d=3$ and $S$ has finite area, then a cusped Hitchin representation is simply the holonomy map of a finite area strictly convex 
projective structure on $S$ (see Marquis \cite{marquis-surface}). More generally, if $\rho\colon\Gamma\to\sf{SL}(3,\mathbb R)$ acts geometrically finitely,
in the sense of Crampon-Marquis \cite[Def. 5.14]{crampon-marquis}, on a strictly convex domain with
$C^1$ boundary, then $\rho$ is cusped Hitchin by \cite[1.3. Thm.]{fock-goncharov}.

Let
$$\mathfrak{a}=\{ \vec a \in\mathbb R^d\ |\ a_1+\cdots+a_d=0\}$$
be the standard Cartan algebra for the Lie algebra $\mathfrak{sl}(d,\mathbb R)$ of $\mathsf{SL}(d,\mathbb R)$.
If $T\in\mathsf{SL}(d,\mathbb R)$, let 
$$\lambda_1(T)\ge\cdots\ge\lambda_d(T)$$ 
be the (ordered) moduli of (generalized) eigenvalues of $T$ (with multiplicity).
The Jordan  (or Lyapunov) projection
$$\ell:\mathsf{SL}(d,\mathbb R)\to\mathfrak{a}\ \ 
\mathrm{is\ given\ by}\  \ell(T)=(\log \lambda_1(T),\cdots,\log\lambda_d(T)).$$

For each $k=1,\ldots,d-1$, let $\alpha_k:\mathfrak{a}\to\mathbb R$ be given by $\alpha_k(\vec a)=a_k-a_{k+1}$ and
let 
$$\Delta=\left\{\sum_{k=1}^{d-1} t_t\alpha_k\ |\ t_k\ge0\ \forall\ k\ \mathrm{and}\ t_k>0\ \mathrm{for\ some}\ k\right\}\subset\mathfrak{a}^*.$$
For example, if $\alpha_H$ is the Hilbert length functional given by $\alpha_H(\vec a)=a_1-a_d$, then $\alpha_H=\sum_{k=1}^{d-1}\alpha_k\in\Delta$.
Similarly, if $\omega_1(\vec a)=a_1$, then
$\omega_1=\sum_{k=1}^{d-1}  \frac{d-k}{d}\alpha_k\in\Delta$.
Given non-trivial $\phi\in\Delta$ and $T\in\mathsf{SL}(d,\mathbb R)$, we define the {\em $\phi$-translation length}
$$\ell^\phi(T)=\phi(\ell(T)).$$

Let $(\Sigma^+,\sigma)$ be the Stadlbauer-Ledrappier-Sarig coding \cite{ledrappier-sarig,stadlbauer} (if $S$ has finite area)
or Dal'bo-Peign\'e coding \cite{dalbo-peigne} (if not)
of the recurrent portion of the geodesic flow on $T^1S$. 
It is topologically mixing and has (BIP). Moreover, it comes equipped with a map 
$$G:\mathcal A\to \Gamma$$
so that if $\gamma\in \Gamma$ is hyperbolic, then there exists
$x=\overline{x_1\cdots x_n}\in\Sigma^+$ so that $G(x_1)\cdots G(x_n)$ is conjugate to $\gamma$.
Moreover, $x$ is unique up to powers of $\sigma$.
Given a cusped Hitchin representation 
$\rho:\Gamma\to\mathsf{SL}(d,\mathbb R)$
we will define a vector-valued roof function $\tau_\rho:\Sigma^+\to \mathfrak{a}$
with the property that if $x=\overline{x_1\cdots x_n}$ is a periodic element of $\Sigma^+$, then
$$S_n\tau_\rho(x)=\tau_\rho(x)+\tau(\sigma(x))+\cdots+\tau_\rho(\sigma^{n-1}(x))=\ell\big(\rho(G(x_1)\cdots G(x_n))\big)$$
so $\tau_\rho$ encodes all the spectral data of $\rho(\Gamma)$.

The following result allows us to use the general thermodynamical machinery we developed to study
cusped Hitchin representations.

\begin{thmx}[Roof functions]
\label{Roof Properties}
Suppose that $\Gamma$ is a torsion-free, geometrically finite  Fuchsian group which is not convex cocompact,
$\rho:\Gamma\to\mathsf{SL}(d,\mathbb R)$ is a cusped Hitchin representation 
and $\phi\in\Delta$. Then there exists a locally H\"older continuous function
$\tau_\rho^\phi=\phi\circ\tau_\rho:\Sigma^+\to\mathbb R$ such that
\begin{enumerate}
\item
$\tau_\rho^\phi$ is eventually positive and non-arithmetic.
\item
If $x=\overline{x_1\cdots x_n}$ is a periodic element of $\Sigma^+$, then
$$S_n\tau_\rho^\phi(x)=\ell^\phi\big(\rho(G(x_1)\cdots G(x_n))\big).$$
\item
$\tau_\rho^\phi$ has a strong entropy gap at infinity. Moreover, if $\phi=a_1\alpha_1+\cdots+a_{d-1}\alpha_{d-1},$
then
$$d(\tau_\rho^\phi)=\frac{1}{2(a_1+\cdots +a_{d-1})}.$$
\item 
If $\eta:\Gamma\to \mathsf{SL}(d,\mathbb R)$ is another cusped Hitchin representation, then there
exists $C>0$ so that
$$|\tau_\rho^\phi(x)-\tau_\eta^\phi(x)|\le C$$
for all $x\in\Sigma^+$. 
\end{enumerate}
\end{thmx}

We obtain a counting result for cusped Hitchin representations as an immediate consequence of
Theorem \ref{thm:countingN}.

\begin{cor}
\label{cusped counting}
If $\rho:\Gamma\to\mathsf{SL}(d,\mathbb R)$ is a cusped Hitchin representation and $\phi\in\Delta$,
then there exists a unique $\delta=\delta_\phi(\rho)$ so that $P(-\delta \tau_\rho^\phi)=0$, and
$$\# \big\{[\gamma]\in[\Gamma]\ \big| \ 0<\ell^\phi(\rho(\gamma))\le t\big\}\sim \frac{e^{t\delta}}{t\delta}$$
as $t\to\infty$.
\end{cor}

We will refer to $\delta_\phi(\rho)$ as the {\em $\phi$-topological entropy} of $\rho$.

If $\rho,\eta:\Gamma\to\mathsf{SL}(d,\mathbb R)$ are cusped Hitchin representations and $\phi\in\Delta$, 
we define the {\em Manhattan curve}
$$\mathcal C^{\phi}(\rho,\eta)=\{ (a,b)\in \mathbb R^2 \ | \ P(-a\tau_\rho^{\phi}-b\tau_\eta^{\phi})=0,\   a\ge 0,\  b\ge 0,\  a+b>0\}.$$
Theorem \ref{Manhattan curve} immediately gives the following information about
$\mathcal C^{\phi}(\rho,\eta)$.

\begin{cor}
\label{CuspedManhattan}
If $\rho,\eta:\Gamma\to\mathsf{SL}(d,\mathbb R)$ are cusped Hitchin representations and $\phi\in\Delta$,
then
\begin{enumerate}
\item
$\mathcal C^{\phi}(\rho,\eta)$  is a closed subsegment of  an analytic curve,
\item
the points $(\delta_{\phi}(\rho),0)$ and $(0,\delta_{\phi}(\eta))$ lie on $\mathcal C^{\phi}(\rho,\eta)$,
\item
and $\mathcal C^{\phi}(\rho,\eta)$  is strictly convex, unless 
$$\ell^{\phi}(\rho(\gamma))=\frac{\delta_\phi(\eta)}{\delta_\phi(\rho)}\ell^{\phi}(\eta(\gamma))$$ 
for all 
$\gamma\in\Gamma$.
\end{enumerate}
Moreover,  the tangent line to $\mathcal C^{\phi}(\rho,\eta)$ at $(\delta_{\phi}(\rho),0)$ has slope
$$s^{\phi}(\rho,\eta)=-\frac{\int \tau_\eta^{\phi} d{\mu}_{-\delta^{\phi}(\rho)\tau^{\phi}_\rho}}{\int \tau_\rho^{\phi}\ 
d{\mu}_{-\delta^{\phi}(\rho)\tau^{\phi}_\rho}}
$$
\end{cor}

We call $I^{\phi}(\rho,\eta)=-s^{\phi}(\rho,\eta)$ the {\em $\phi$-pressure intersection}. We also define the
{\em renormalized \hbox{$\phi$-pressure} intersection} by
$$J^{\phi}(\rho,\eta)=\frac{\delta_{\phi}(\eta)}{\delta_{\phi}(\rho)}I^{\phi}(\rho,\eta).$$
As a further corollary of Theorem \ref{Manhattan curve} we obtain the following rigidity result for renormalized pressure intersection.
This corollary will later play a key role in our forthcoming construction of pressure metrics on the space of
cusped Hitchin representations.

\begin{cor}
\label{intersection rigidity}
If $\rho,\eta:\Gamma\to\mathsf{SL}(d,\mathbb R)$ are cusped Hitchin representations and $\phi\in\Delta$, then
$$J^{\phi}(\rho,\eta)\ge 1$$
with equality if and only if 
$$\ell^{\phi}(\rho(\gamma))=\frac{\delta_\phi(\eta)}{\delta_\phi(\rho)}\ell^{\phi}(\eta(\gamma))$$
for all $\gamma\in\Gamma$. 
\end{cor}

As a corollary of Theorem \ref{thm:equid_roof} we obtain the following geometric interpretation of the
pressure intersection. Let
$$R_T^\phi(\rho)= \big\{[\gamma]\in[\Gamma]\ \big| \ 0< \ell^\phi(\rho(\gamma))\le T\big\}.$$

\begin{cor}
\label{geometric intersection}
If $\rho,\eta:\Gamma\to\mathsf{SL}(d,\mathbb R)$ are cusped Hitchin representations and $\phi\in\Delta$
then
$$I^{\phi}(\rho,\eta)=\lim_{T\to\infty} \frac{1}{\#(R_T^{\phi}(\rho))}\sum_{[\gamma]\in R_T^{\phi}(\rho)}
\frac{\ell^{\phi}(\eta(\gamma))}{\ell^{\phi}(\rho(\gamma))}.$$
\end{cor}

In a companion paper, Canary, Zhang and Zimmer \cite{CZZ} study the geometry of cusped Hitchin representation
showing that they are ``relatively'' Borel Anosov in a sense which generalizes work of Labourie \cite{labourie-invent}.
They also show that cusped Hitchin representations are stable with respect to type-preserving deformation in $\mathsf{SL}(d,\mathbb C)$.
As a consequence, they see that limit maps are H\"older and vary analytically. In \cite{BCKM2}, we combine
the work in this paper and in \cite{CZZ} to construct pressure metrics on cusped Hitchin components.

This project is motivated by the hope that there is a geometric theory of the augmented Hitchin component which generalizes the classical theory for 
augmented Teichm\"uller space. Masur \cite{masur-wp} proved that the augmented Teichm\"uller space is
the metric completion of Teichm\"uller space with the Weil-Petersson metric. The strata at infinity of
augmented Teichm\"uller space consists of Teichm\"uller spaces of cusped hyperbolic surfaces. These strata
naturally inherit a Weil-Petersson metric from the completion.  The potential analogy is clearest when $d=3$, where
Hitchin components are spaces of convex projective structures on closed surfaces. Work of Loftin \cite{loftin-bord}
and Loftin-Zhang \cite{loftin-zhang} explores the analytic structure and topology of this  bordification.
We hope that our work on pressure metrics will aid in showing that there is an augmented Hitchin component
which arises as the metric completion of the  Hitchin component with the pressure metric. See the survey paper \cite{canary-survey} for
a more detailed discussion of the conjectural picture.

\medskip\noindent
{\bf Other applications:} These results  have immediate generalizations for $P_k$-Anosov representations
of geometrically finite Fuchsian groups.

We  also recover (mild generalizations of) many of Sambarino's results  on counting and equidistribution
for uncusped
Anosov representations in our framework (see \cite{sambarino-quantitative,sambarino-indicator,sambarino-orbital}).

\medskip\noindent
{\bf Historical remarks:} 
Counting and equidistribution results have long been a central theme of the Thermodynamical Formalism
(see, for example, the seminal work of Bowen, Parry, Pollicott and Ruelle \cite{bowen,bowen-ruelle,parry-pollicott,ruelle}).
Lalley's innovation \cite{lalley} was the introduction of renewal theory and  the development of a Renewal Theorem which allowed him to 
obtain precise counting and equidistribution results. Our work harnesses Kesseb\"ohmer and Kombrink's extension \cite{kess-kom}  
of Lalley's Renewal Theorem to the setting of countable Markov shifts to obtain similar results in our setting.

Bishop and Steger \cite{bishop-steger}  proved a rigidity theorem  in the setting of finite area hyperbolic surfaces
which is the precursor to the study of Manhattan curves. Lalley \cite{lalley-manhattan} extended  Bishop and Steger's
rigidity theorem to the setting of closed negatively curved surfaces.
The formulation in terms of a Manhattan curve is due to
 Burger \cite{burger} who worked in the setting of convex cocompact representations into rank one Lie groups.
 Kao \cite{kao-manhattan} established a Manhattan curve theorem for geometrically finite Fuchsian groups and
 Bray-Canary-Kao \cite{BCK} extended his result to the setting of geometrically finite quasifuchsian representations.
 
 Dal'bo and Peign\'e \cite{dalbo-peigne} used renewal theorems in their work obtaining counting and mixing results on geometrically finite
 negatively curved surfaces. They also applied renewal techniques to study counting results for the modular surface
 \cite{dalbo-peigne-modular}.  Thirion \cite{thirion} used related techniques to obtain asymptotic results for orbital counting
 functions for ping pong groups.  Thirion's ping pong groups overlap with the class of (images of) cusped $P_1$-Anosov representations.
 
 Corollary \ref{cusped counting} generalizes results of Sambarino \cite{sambarino-quantitative,sambarino-indicator,sambarino-orbital} from
 the Anosov setting, while Corollaries \ref{intersection rigidity} and \ref{geometric intersection} generalize results of Bridgeman-Canary-Labourie-Sambarino \cite{BCLS}.

 In the case of cusped Hitchin representations, $d(\tau_\rho^\phi)$ is simply the maximum critical exponent of the 
 $\phi$-length Poincar\'e series associated to any unipotent subgroup of $\rho(\Gamma)$. Thus, having a strong entropy
 gap at infinity is analogous to the critical exponent gap used in the work of Dal'bo-Peign\'e \cite{dalbo-peigne} and Dal'bo-Otal-Peign\'e \cite{DOP}.
 Schapira and Tapie \cite[Prop. 7.16]{schapira-tapie} showed that for a geometrically finite negatively curved manifold then there
 is a critical exponent gap if and only if the geodesic flow has an entropy gap at infinity. Our definition is inspired by their work. 
 In turn, Schapira and Tapie were motivated, in part, by work on strongly positive recurrent potentials for countable Markov shifts due to  
 Gurevich and Savchenko \cite{gur-sav,savchenko}, Sarig \cite{sarig-first,sarig-phase}, Ruette \cite{ruette}, and Boyle-Buzzi-G\'omez \cite{BBG}. Other
 relevant precursors to our results include the work of  Iommi-Riquelme-Velozo \cite{IRV}, Riquelme-Velozo \cite{riquelme-velozo}, and Velozo \cite{velozo}.

In recent work, Pollicott and Urbanski \cite{pollicott-urbanski} use related techniques to
obtain fine counting results for conformal dynamical systems.
Their main technical tools come from the study of complexified  Ruelle-Perron-Frobenius operators, generalizing
early work of Parry-Pollicott \cite{parry-pollicott} in the setting of finite Markov shifts.
(The proof of Kesseb\"ohmer and Kombrink's Renewal Theorem \cite{kess-kom} also relies on the study of
complexified Ruelle-Perron-Frobenius operators.)
Pollicott and Urbanski give extensive applications to the study of circle packings, rational functions, continued fractions, Fuchsian groups and
Schottky groups and other topics.

Feng Zhu \cite{feng-hilbert} obtained closely related counting and equidistribution results for the Hilbert length functional on geometrically\
finite strictly convex projective manifolds. When $d=3$, cusped Hitchin representations are holonomy maps of
strictly convex projective surfaces, so our results overlap with his in this case.

\medskip\noindent
{\bf Outline of paper:} In Section 2, we recall the relevant background material from the theory of countable Markov shifts. In Section 3,
we use this theory to explore the consequences of entropy gaps at infinity. In Section 4, we recall the Renewal Theorem of
Kesseb\"ohmer and Kombrink \cite{kess-kom} and show that we can apply it in our context. Section 5 contains the crucial
technical material needed in the proof of Theorems A. Sections 6, 7 and 8 contain the proof of Theorems A, B and C
(respectively). In Section 9, we develop the background material needed for our applications. Section 10 contains the proof
of (a generalization of) Theorem D and Section 11 derives its consequences.

\medskip\noindent
{\bf Acknowledgements:} The authors would like to thank Godofredo Iommi,  Andres Sambarino,  Barbara Schapira, Ralf Spatzier and Dan Thompson for helpful
comments and suggestions. We also thank the referee for suggestions which improved the exposition.

This material is based upon work supported by the National Science Foundation
under Grant No. DMS-1928930 while the second author participated in a program hosted
by the Mathematical Sciences Research Institute in Berkeley, California, during
the Fall 2020 semester.

\section{Background from the Thermodynamic Formalism}
\label{countable background}

In this section, we recall the background results we will need from the Thermodynamic Formalism for countable Markov shifts
as developed by Gurevich-Savchenko \cite{gur-sav}, Mauldin-Urbanksi \cite{MU} and Sarig \cite{sarig-first}.

Given a countable alphabet $\mathcal A$ and a transition matrix $\mathbb T=(t_{ab})\in\{0,1\}^{\mathcal A\times\mathcal A}$
a one-sided Markov shift is
$$\Sigma^+=\{x=(x_i)\in\mathcal A^{\mathbb N}\ |\ t_{x_ix_{i+1}}=1\ {\rm for}\ {\rm all}\ i\in\mathbb N\}$$
equipped with a shift map $\sigma:\Sigma^+\to\Sigma^+$ which takes $(x_i)_{i\in\mathbb N}$ to
$(x_{i+1})_{i\in\mathbb N}$. 

We will work in the setting of topologically mixing Markov shifts with (BIP), where many of the classical
results of Thermodynamic Formalism generalize. The shift $(\Sigma^+,\sigma)$ is {\em topologically mixing} if for all 
$a,b\in \mathcal A$, there exists $N=N(a,b)$
so that if $n\ge N$, then there exists $x\in\Sigma$ so that $x_1=a$ and $x_n=b$.  It has the
big images and pre-images property (BIP) if there exists a finite subset $\mathcal B\subset\mathcal A$ so
that if $a\in\mathcal A$, then there exist $b_0,b_1\in\mathcal B$ so that $t_{b_0a}=1=t_{ab_1}$.

The theory works best for locally H\"older continuous potentials.
We say that  $g:\Sigma^+\to\mathbb R$ is {\em locally H\"older continuous}  if there exist $A>0$ and $\alpha>0$ 
so that
$$|g(x)-g(y)|\le Ae^{-\alpha n}$$
whenever $x_i=y_i$ for all $i\le n$ and $n\in\mathbb N$. When we want to record the constants we will say that
$g$ is locally $\alpha$-H\"older continuous with constant $A$.
The {\em Gurevich pressure} of $g$ is given by
$$P(g)=\lim_{n\to\infty}\frac{1}{n}\log \sum_{\{x\in \mathrm{Fix}^n\ |\ x_1=a\}} e^{S_ng(x)}$$
for some (any)  $a\in\mathcal A$ where 
$$S_ng(x)=\sum_{i=1}^{n} g(\sigma^{i-1}(x))$$
is the {\em ergodic sum}
and $\mathrm{Fix}^n=\{x\in\Sigma^+\ |\ \sigma^n(x)=x\}$.

We say that two locally H\"older continuous functions $f$ and $g$ are
{\em cohomologous} if there exists a locally H\"older continuous function $h$ so that
$$f-g=h-h\circ\sigma.$$
The analogue of Livsic's theorem holds in this setting.

\begin{thm} {\rm (Sarig  \cite[Thm 1.1]{sarig-2009})}
\label{livsic}
Suppose that $\Sigma^+$ is a topologically mixing, one-sided countable Markov shift with (BIP).
If $f:\Sigma^+\to\mathbb R$  and $g:\Sigma^+\to\mathbb R$ are both locally H\"older continuous,
then $f$ is cohomologous to $g$ if and only if
$S_nf(x)=S_ng(x)$ for all $n\in\mathbb N$ and $x\in\mathrm{Fix}^n$. In particular, if $f$ and $g$ are cohomologous,
then $P(-tf)=P(-tg)$ whenever $P(-tf)$ is finite.
\end{thm}

A $\sigma$-invariant Borel probability measure $\mu$ on $\Sigma^+$ is an {\em equilibrium state} for 
a locally H\"older continuous function $g:\Sigma^+\to\mathbb R$ if 
$$P(g)=h_\sigma(\mu)+\int_{\Sigma^+} g \ d\mu$$
where $h_\sigma(\mu)$ is the measure-theoretic entropy of $\sigma$ with respect to the measure $\mu$. 

A Borel probability measure $\mu$ on $\Sigma^+$ is a {\em Gibbs state} for 
a locally H\"older continuous function $g:\Sigma^+\to\mathbb R$ if there exists  $B>1$ so that
$$\frac{1}{B}\le\frac{\mu([a_1,\ldots,a_n])}{e^{S_ng(x)-nP(g)}}\le B$$ 
for all $x\in [a_1,\ldots,a_n]$, where $[a_1,\ldots,a_n]$ is the {\em cylinder} consisting of all $x\in\Sigma^+$ so
that $x_i=a_i$ for all $1\le i\le n$. 

\begin{thm} {\rm (Mauldin-Urbanski \cite[Thm 2.2.9]{MU}, Sarig \cite[Thm 4.9]{sarig-2009})} 
\label{Gibbs is eq}
If $\Sigma^+$ is a topologically mixing, one-sided countable Markov shift with (BIP), $g:\Sigma^+\to\mathbb R$ is 
locally H\"older continuous, it admits a shift invariant Gibbs state $\mu_g$, and $-\int g\ d\mu_g<+\infty$, then
$\mu_g$ is the unique equilibrium state for $g$.
\end{thm}

Recall from the introduction that for $g\colon\Sigma^+\to \mathbb R$ a locally H\"older continuous function we define
$$I(g,a)=\inf\big\{ g(x)\ |\ x\in\Sigma^+, x_1=a\big\}\qquad\mathrm{and}\qquad S(g,a)=\sup\big\{ g(x)\ |\ x\in\Sigma^+, x_1=a\big\}.$$

We will make crucial use of the following criterion for a potential to admit an equilibrium state.

\begin{thm}
\label{equnique}
{\rm (Mauldin-Urbanski \cite[Thm 2.2.4 and 2.2.9, Lemma 2.2.8]{MU}, Sarig \cite[Thm 4.9]{sarig-2009})} 
If $\Sigma^+$ is a topologically mixing, one-sided countable Markov shift with (BIP),
$g:\Sigma^+\to\mathbb R$  is locally H\"older continuous, and
$$\sum_{a\in\mathcal A} I(g,a) e^{-S(g,a)}$$ 
converges, then $-g$ admits a unique equilibrium state $\mu_{-g}$.
Moreover, 
$$\int_{\Sigma^+} g\ d\mu_{-g}<+\infty.$$
\end{thm}

We will need to be able to take the derivatives of the pressure function and to be able to apply the Implicit Function
Theorem. 
We say that $\{g_u:\Sigma^+\to \mathbb R\}_{u\in M}$ is a {\em real analytic family} if $M$ is a real analytic
manifold and for all $x\in \Sigma^+$, $u\to g_u(x)$ is a real analytic function on $M$. 
Mauldin and Urbanski \cite[Thm. 2.6.12,\ Prop. 2.6.13]{MU}
(see also Sarig \cite[Cor. 4]{sarig-2003}), prove real analyticity properties of the pressure function and evaluate its derivative.

\begin{thm}{\rm (Mauldin-Urbanski, Sarig)}
\label{pressure analytic}
Suppose that $\Sigma^+$ is a topologically mixing, one-sided countable Markov shift with (BIP).
If $\{g_u:\Sigma^+\to \mathbb R\}_{u\in M}$ is a real analytic family of locally H\"older continuous functions such
that $P(g_u)<\infty$ for all $u$, then $u\to P(g_u)$ is real analytic.

Moreover, if $v\in T_{u_0}M$ and there exists a neighborhood $U$ of $u_0$ in $M$ so that if $u\in U$ and $-\int_{\Sigma^+} g_u d\mu _{g_{u_0}}<\infty$, then
$$D_vP(g_u)=\int_{\Sigma^+} D_v(g_u(x))\  d\mu_{g_{u_0}}.$$
\end{thm}

Recall that if $f:\Sigma^+\to\mathbb R$ is locally H\"older continuous the {\em transfer operator} is defined by
\[
\mathcal{L}_{f}\phi(x):=\sum_{y\in\sigma^{-1}(x)}e^{f(y)}\phi(y)
\]
where  $\phi:\Sigma^+\to\mathbb R$ is a \textit{bounded} locally
H\"older continuous function. 
The transfer operator, in particular, gives us crucial information about equilibrium states.

\begin{thm} {\rm (Mauldin-Urbanski \cite[Cor. 2.7.5]{MU}, Sarig \cite[Thm. 4.9]{sarig-2009})}
\label{transfer fact}
Suppose that $\Sigma^+$ is a topologically mixing, one-sided countable Markov shift with (BIP).
If $g :\Sigma^+\to\mathbb R$ is locally H\"older continuous, \hbox{$P(g)<+\infty$}, and $\sup g<+\infty$ then there
exist unique probability measures $\mu_g$ and $\nu_g$ on $\Sigma^+$  and a positive function
$h_g:\Sigma^+\to\mathbb R$
so that
$$\mu_g=h_g\nu_g,\qquad \mathcal L_gh_g=e^{P(g)}h_g,\qquad\mathrm{and}\qquad
\mathcal L_g^*\nu_g=e^{P(g)}\nu_g.$$
Moreover, $h_g$ is bounded away from both $0$ and $+\infty$ and $\mu_g$ is an equilibrium state for $g$.
\end{thm}

We will also use the following estimate  on the behavior of powers of the transfer operator.

\begin{thm}{\rm (Mauldin-Urbanski \cite[Theorem 2.4.6]{MU})} 
\label{bounds on transfer}
Suppose that $\Sigma^+$ is a topologically mixing, one-sided countable Markov shift with (BIP).
If $g:\Sigma^+\to\mathbb R$ is locally H\"older continuous, 
\hbox{$P(g)<+\infty$}, and $\sup g<+\infty$, then there exist $R>0$ and $\eta\in(0,1)$ so that
if $n\in\mathbb N$ and $\phi:\Sigma^+\to\mathbb R$ is bounded and locally $\eta$-H\"older continuous with constant $A$, then
\begin{equation}
\Big\| e^{-nP(g)}\mathcal L_g^n\phi-h_{g}(x)\int \phi\ d\nu_{g}\Big\| \le R\eta^n\Big(\sup_{x\in\Sigma^+}|\phi(x)|+A\Big).
\end{equation}
\end{thm}

\section{Entropy gaps at infinity}
\label{entropy gaps}

In this section, we show that a strong entropy gap at infinity implies a weak entropy gap at infinity
and explore the thermodynamical consequences of entropy gaps at infinity.

Recall that $d(f)$ is the critical exponent of the series 
$$Z_1(f,s)=\sum_{a\in\mathcal A} e^{-sS(f,a)}.$$
Notice that if  $f$ is locally H\"older continuous, there exists $C>0$
so that
$S(f,a)- I(f,a)\le C$
for all $a\in\mathcal A$. So the series
$$\sum_{a\in\mathcal A} e^{-sI(f,a)}$$
has critical exponent $d(f)$ and diverges at $d(f)$
if and only if $f$ has a strong entropy gap at infinity.

We first observe a bound on the number of letters with $I(f,a)\le t$.

\begin{lem}
\label{growth of alphabet}
Suppose that $\Sigma^+$ is a topologically mixing, one-sided countable Markov shift with (BIP).
If $f:\Sigma^+\to\mathbb R$ is locally H\"older continuous, $d(f)$ is finite and $b>d(f)$, then there exists $D=D(f,b)>0$ so that
$$B_1(f,t)=\#\big\{a\in\mathcal A\ |\ I(f,a) \le t\big\}\le De^{bt}$$
for all $t>0$, and
$$\sum_{y\in\sigma^{-1}(x)}{\bf 1}_{\{f(y)\leq t\}}(y)\le De^{bt}$$
for all $x\in\Sigma^+$ and  $t>0$.
\end{lem}

\begin{proof}
Fix $b>d(f)$. If there does not exist $D$ so that $B_1(f,t)\le De^{bt}$ for all $t>0$, then there exists a sequence $t_n\to\infty$ so that 
$$B_1(f,t_n)\ge ne^{bt_n}.$$
But then 
$$\sum_{a\in\mathcal A} e^{-bI(f,a)}\ge \sum_{\{a\ |\  I(f,a)\le t_n \}} e^{-bI(f,a)}\ge n e^{bt_n}e^{-bt_n}=n$$
for all $n\in\mathbb N$, which contradicts our assumption that $b>d(f)$.

Finally, notice that if $x\in\Sigma^+$, then
$$\sum_{y\in\sigma^{-1}(x)}{\bf 1}_{\{f(y)\leq t\}}(y)\le B_1(f,t)\le De^{bt}$$
for all $t>0$.
\end{proof}

It will often be convenient to work with a strictly positive potential. We observe that an eventually positive potential is
always cohomologous to a strictly positive potential with the same entropy gaps.

\begin{lem}
\label{eventually to strictly}
Suppose that $\Sigma^+$ is a topologically mixing, one-sided countable Markov shift with (BIP)
and that $f:\Sigma^+\to\mathbb R$ is eventually positive, locally H\"older continuous and $d(f)$ is finite. 
Then $f$ is cohomologous to a strictly positive, locally H\"older continuous function $g$ so that
\begin{enumerate}
\item
there exists $C$ so that $|f(x)-g(x)|\le C$ for all $x\in\Sigma^+$,
\item
$d(f)=d(g)$,
\item 
$f$  has a weak entropy gap at infinity if and only if $g$ has a weak entropy gap at infinity, and
\item
$f$  has a strong entropy gap at infinity if and only if $g$ has a strong entropy gap at infinity.
\end{enumerate}
\end{lem}

\begin{proof} 
Notice that  (1) implies that $|S(f,a)-S(g,a)|\le C$. 
Moreover, if $f$ is cohomologous to $g$, and both are locally H\"older continuous, then $P(-tf)=P(-tg)$ for all $t>d(f)$, see Theorem \ref{livsic}.
Therefore, (2)--(4) follow immediately  once we  construct a strictly positive, locally H\"older continuous function $g$
 that is cohomologous to $f$
so
that (1) holds.

Let 
\[
  R=\left|\inf_{x\in\Sigma^+} f(x)\right|.
\]
Note that $R=|\inf_{a\in\mathcal A} I(f,a)|$ is finite since there exists $s>d(f)>0$ so that $\sum_{a\in\mathcal A} e^{-s I(f,a)}$ is finite. 
Since $f$ is eventually positive, there exists $N\in\mathbb N$ and $B>0$ so
that if $n\ge N$ and $x\in\Sigma^+$, then
$$S_nf(x)\ge B.$$
Let
$$\mathcal F=\{a\in\mathcal A\ |\ I(f,a)\le RN+B\}.$$
Since $d(f)$ is finite, $\mathcal F$ must be finite.
To see this, observe that for $s>d(f)>0$
\[
\infty>\sum_{a\in\cal A}e^{-s I(f,a)}\geq \sum_{a\in\cal F}e^{-sI(f,a)}\geq 
  \sum_{a\in\mathcal F} e^{-s(RN+B)}.
\]

For all $n\in\mathbb N$, define
\begin{align*}C_nf(x)&=\sum_{i=1}^{n} \Big(f(\sigma^{i-1}(x)) {\bf 1}_{\{x_i\in\mathcal  F\}}(x)+(RN+B){\bf 1}_{\{x_i\notin\mathcal F\}}(x)\Big)\\
&=S_nf(x)-\sum_{i=1}^n\Big(f(\sigma^{i-1}(x))-(RN+B)\Big)\mathbf{1}_{\{x_i\not\in\cal
F\}}(x).
\end{align*}
By construction, 
$$RN^2+NB+TN\ge C_Nf(x)\ge B$$
for all $x\in\Sigma^+$, where
$$T=\sup\{f(x)\ |\ x_1\in\mathcal F\}.$$
(The lower bound holds, since $C_Nf(x)=S_Nf(x)\ge B$ if $x_i\in\mathcal F$ for all $i\le N$, and otherwise one of the summands of $C_Nf(x)$ is $RN+B$
and each of the remaining terms are bounded below by $-R$.)

We then define $g:\Sigma^+\to\mathbb R$ by
$$g(x)=\frac{1}{N}C_Nf(x)+\big(f(x)-(RN+B)\big)\mathbf{1}_{\{x_1\not\in\cal F\}}(x).
$$
By construction, $g$ is continuous and
$$g(x)\ge \frac{B}{N}>0$$
for all $x\in\Sigma^+$, so $g$ is strictly positive.

Moreover, if $x_1\in\mathcal F$, then $|g(x)-f(x)|\leq RN+B+2T$,
and if
$x_1\notin\mathcal F$, 
then 
$$|g(x)-f(x)|\leq RN+B+\frac{1}{N} C_Nf(x)\leq 2(RN+B).$$
It follows that
$$|g(x)-f(x)|\le 2(RN+B+T)=:C
$$
for all $x\in\Sigma^+$. 

To show $g$ is locally H\"older continuous, consider $x,y\in\Sigma^+$ for which $x_i=y_i$ for all $i=1,\ldots, n$,
and note that it suffices to consider  $n\geq N$. Then 
  \begin{align*}
    |g(x)-g(y)| = \left|\frac1N \left( \sum_{i=1}^N
    (f(\sigma^{i-1}(x))-f(\sigma^{i-1}(y)))\mathbf 1_{\{x_i\in\mathcal
  F\} }(x) \right) + (f(x)-f(y))\mathbf 1_{\{ x_1\not\in\mathcal F\} }(x)\right|. 
  \end{align*}
  Since $n\geq N,$ applying local H\"older continuity of $f$ gives the
  desired conclusion.

Finally, if $x=\overline{x_1\ldots x_r}\in\mathrm{Fix}^r$, then one may check that
$S_rf(x)=S_rg(x)$. 
To see this, observe that
\begin{align*}
S_rg(x)&=S_r\left(\frac{1}{N}C_Nf(x)\right)+S_r\Big(\left(f(x)-(RN+B)\right)\textbf1_{\{x_1\not\in\cal F\}}(x)\Big)\\
&=\frac{1}{N}S_rC_N f(x)+\sum_{j=1}^r\left(f(\sigma^{j-1}(x))-(RN+B)\right)\textbf1_{\{x_j\not\in\cal F\}}(x)
\end{align*}
and since $\sigma^r(x)=x$,
\begin{align*}
S_rC_Nf(x)&=S_rS_Nf(x)-\sum_{j=1}^r\sum_{i=1}^N\left(f(\sigma^{i-1}(x))-(RN+B)\right)\textbf1_{\{x_i\not\in\cal F\}}(x)\\
&=NS_rf(x)-N\sum_{j=1}^r\left(f(\sigma^{j-1}(x))-(RN+B)\right)\textbf1_{\{ x_j\not\in\cal F\}}(x).
\end{align*}

Theorem \ref{livsic} then implies that $f$ and $g$ are cohomologous.
\end{proof}

We next study the behavior of $P(-tf)$ for $t>d(f)$, showing among other things that a strong entropy gap at infinity implies a weak entropy gap at infinity.

\begin{lem}
\label{gap and pressure}
Suppose that $\Sigma^+$ is a topologically mixing, one-sided countable Markov shift with (BIP) and
 $f:\Sigma^+\to\mathbb R$ is locally H\"older continuous and eventually positive.
\begin{enumerate}
\item
If $d(f)$ is finite, then $P(-tf)$ is finite if $t>d(f)$ and infinite if $t<d(f)$, and the function $t\to P(-tf)$ is monotone decreasing 
and analytic on $(d(f),\infty)$. 
\item
There exists at most one  $\delta\in (d(f),\infty)$  so 
that $P(-\delta f)=0$.
\item
If $f$ has a strong entropy gap at infinity, then $t\to P(-tf)$ is proper
on $(d(f),\infty)$.
In particular, $f$ has a weak entropy gap at infinity.
\end{enumerate}
\end{lem}
 
\begin{proof}
Mauldin and Urbanski \cite[Theorem 2.1.9]{MU} proved that if $\Sigma^+$ is topologically mixing and
has (BIP), then $P(-sf)$ is finite if and only  if
$$
Z_1(-f,s)=\sum_{a\in\cal A}e^{\sup\{-sf(x)\ \mid\ x_1=a\}}
$$ 
converges. Therefore, $P(-tf)$ is finite if $t>d(f)$ and infinite if $t<d(f)$.
Notice that $t\to P(-tf)$ is monotone decreasing by definition and analytic by Theorem \ref{pressure analytic},
so (1) follows. (2) is an immediate consequence of (1).
 
It remains to show (3).  The fact that $\lim_{t\to d(f)} P(-tf)=+\infty$ is essentially contained in
Mauldin and Urbanski's proof of \cite[Theorem 2.1.9]{MU}, but we elaborate here for completeness. 
They show that 
there exist constants \hbox{$q,s,M,m>0$} so that for any locally
{H\"older} continuous function $g$, 
$$
\sum_{i=n}^{n+s(n-1)}Z_{i}(g,1)\geq  \frac{e^{-M+(M-m)n}}{q^{n-1}}Z_{1}(g,1)^{n}.
$$
where 
$$Z_n(g,1)=\sum_{p\in\Lambda_k}e^{\sup_{x\in p}S_ng(x)},$$
and $\Lambda_k$ is the set of $k$-cylinders of $\Sigma^+$. 
They observe \cite[Equation (2.1)]{MU} that \hbox{$\lim\frac{1}{n}\log Z_n(g,1)=P(g)$.}
Thus there exists $A>0$ such that for all $n$, there exists  $\hat n\in [n,n+s(n-1)]$ so that
$Z_{\hat n}(g,1)\ge A^nZ_1(g,1)^n$, so $P(g)\ge \frac{1}{1+s}\log AZ_1(g,1)$.
Therefore, if $f$ has a strong entropy gap at infinity, then $\lim_{t\to d(f)} Z_1(-tf,1)=+\infty$ and hence 
$$\lim_{t\to d(f)} P(-tf)\ge\lim_{t\to d(f)} \frac{1}{1+s}\log AZ_1(-tf,1)=+\infty.
$$

We now show that ${\displaystyle \lim_{t\to\infty}P(-tf)=-\infty}$. Notice
that since there exists $N>0$ such that \hbox{$S_{n}f(x)>B>0$} for all $n\ge N$ and $x\in\Sigma^+$, we have $S_{kN}f(x)>kB$ for every $k\geq 1$. Then,
	\[
	\sum_{\{x\in\mathrm{Fix}^{kN}|\ x_{1}=a\}}e^{-2td(f)S_{kN}f(x)}\leq\sum_{\{x\in\mathrm{Fix}^{kN}|\ x_{1}=a\}}e^{-2(t-1)d(f)kB-2d(f)S_{kN}f(x)}
	\]
	which implies 
	\begin{align*}P(-2td(f)f)
		&\leq \lim_{k\to \infty}\frac{1}{kN}\log\sum_{\{x\in\mathrm{Fix}^{kN}|\ x_{1}=a\}}e^{-2(t-1)d(f)kB-2d(f)S_{kN}f(x)}\\
	&= \frac{-2(t-1)d(f)B}{N}+P(-2d(f)f)
	\end{align*}
and so ${\displaystyle \lim_{t\to\infty}P(-tf)=-\infty}$.

Since $t\to P(-tf)$ is proper and monotone decreasing on $(d(f),\infty)$, it follows that
there exists $\delta>d(f)$ so that $P(-\delta f)=0$. Therefore, $f$ has a weak entropy gap at infinity
and we have established (3).
\end{proof}

We next observe that  $-tf$ admits an equilibrium state if $t>d(f)$.

\begin{lem}
\label{equilibrium state exists}
Suppose that $\Sigma^+$ is a topologically mixing, one-sided countable Markov shift with (BIP).
If $f:\Sigma^+\to\mathbb R$ is locally H\"older continuous and eventually positive and
$t>d(f)$,  then there exists a unique equilibrium state 
$\mu_{-t f}$ for $-t f$. 
Moreover,
$$0<\int_{\Sigma^+} f \ d\mu_{-t f}<+\infty.$$
\end{lem}

\begin{proof}
Theorem \ref{equnique} implies that there exists a unique equilibrium state for $-t f$ if and only if
$$\sum_{a\in\mathcal A}  t I(f,a)e^{-tS(f,a)}<+\infty.$$
Indeed, this series converges since
$$\sum_{a\in\mathcal A}e^{-s S(f,a)}<+\infty$$
for all $s>d(f)$. Theorem \ref{equnique} also ensures that  $\int_{\Sigma^+} f d\mu_{-tf}<+\infty$. Since $f$ is eventually  positive,
it is cohomologous to a strictly positive function $g$. 
Then $-tf$ and
$-tg$ are cohomologous and hence 
have the same integral with respect to any
shift-invariant measure, and also 
share the same shift-invariant
equilibrium state, i.e. $\mu_{-tf}=\mu_{-tg}$ (see
\cite[Theorem 2.2.7]{MU} and Theorem \ref{equnique}).
Hence, 
$$\int_{\Sigma^+} f \ d\mu_{-t f}=\int_{\Sigma^+} g \ d\mu_{-t g}>0.$$
\end{proof}

Theorem \ref{transfer fact} and Lemma \ref{gap and pressure}  have the  following corollary which we will use repeatedly.

\begin{cor}
\label{equilibrium and transfer}
Suppose that $\Sigma^+$ is a topologically mixing, one-sided countable Markov shift with (BIP).
If  $f:\Sigma^+\to\mathbb R$ is locally H\"older continuous, eventually positive, and has a weak entropy gap at infinity and $t>d(f)$, then there exist unique probability measures
$\mu_{-tf}$ and $\nu_{-tf}$ on $\Sigma^+$  and a positive function
$h_{-tf}:\Sigma^+\to\mathbb R$
so that
$$\mu_{-tf}=h_{-tf}\nu_{-tf},\qquad \mathcal L_{-tf}h_{-tf}=e^{P(-tf)}h_{-tf},\qquad\mathrm{and}\qquad
\mathcal L_{-tf}^*\nu_{-tf}=e^{P(-tf)}\nu_{-tf}$$
and $h_{-tf}$ is bounded away from both $0$ and $+\infty$.
Moreover, $\mu_{-tf}$ is the equilibrium state of $-tf$.
\end{cor}

We will need analogues of these results for functions of the form $-zg-\delta f$ where $g$ is comparable to $f$ and
$z$ is close to 0.

\begin{prop}
\label{adding g in}
Suppose that $\Sigma^+$ is a topologically mixing, one-sided countable Markov shift with (BIP), 
$f:\Sigma^+\to\mathbb R$ is locally H\"older continuous, eventually positive and has a weak entropy gap at infinity and $P(-\delta f)=0$ for
$\delta=\delta(f)>d(f)>0$.
If $g:\Sigma^+\to\mathbb R$ is locally H\"older continuous, eventually positive, and there exists $C$ so that
$|f(x)-g(x)|\le C$ for all $x\in\Sigma^+$, then
\begin{enumerate}
\item
if $z>d(f)-\delta$, then 
$P(-zg-\delta f)$ is finite, $z\to P(-zg-\delta f)$ is monotone decreasing and analytic on $(d(f)-\delta,\infty)$
 and $\sup_{x\in\Sigma^+}( -zg-\delta f)<+\infty$.
\item
if $z>d(f)-\delta$, then
there exist unique probability measures $\mu_{-zg-\delta f}$ and $\nu_{-zg-\delta f}$ on $\Sigma^+$  and a positive function
$h_{-zg-\delta f}:\Sigma^+\to\mathbb R$
so that
\begin{gather*}\mu_{-zg-\delta f}=h_{-zg-\delta g}\nu_{-zg-\delta f}, \quad \mathcal L_{-zg-\delta f}h_{-zg-\delta f}=e^{P(-zg-\delta f)}h_{-zg-\delta ff},\\  \mathrm{and}\ 
\mathcal L_{-zg-\delta f}^*\nu_{-zg-\delta f}=e^{P(-zg-\delta f)}\nu_{-zg-\delta f}.
\end{gather*}
Moreover, $h_{-zg-\delta f}$ is bounded away from both $0$ and $+\infty$ and
$\mu_{-zg-\delta f}$ is the unique  equilibrium state of $-zg-\delta f$.
\end{enumerate}
\end{prop}

\begin{proof}
Notice that 
$$\sum_{\{x\in \mathrm{Fix}^n\ |\ x_1=a\} } e^{S_n(-zg-\delta f)}\le 
\sum_{\{x\in \mathrm{Fix}^n\ |\ x_1=a\} } 
e^{nzC}
e^{S_n\big( (-z-\delta)f\big)}$$
so $P(-zg-\delta f)$ is finite if $z+\delta>d(f)$, i.e. if $z>d(f)-\delta$. 
Similarly, if $x\in\Sigma^+$, then 
$$(-zg-\delta f)(x)\le
-(z+\delta)f(x)+Cz\le\sup(-(z+\delta)f)+Cz
<+\infty$$
if $z+\delta>0$. The function $z\to P(-zg-\delta f)$ is monotone decreasing by definition and analytic by Theorem \ref{pressure analytic}. We have established (1).

(2) is then an immediate consequence of (1) and  Theorem \ref{transfer fact}.
\end{proof}

\section{Renewal Theorems}

Our main tool will be the Renewal Theorem of Kesseb\"ohmer and Kombrink \cite{kess-kom}. Their result
generalized a result of Lalley  \cite{lalley} for finite Markov shifts. 

Consider a locally H\"older continuous potential $f\colon \Sigma^+\to\mathbb R$. If $\phi:\Sigma^+\to\mathbb R$ is a non-negative, bounded, locally H\"older continuous
function, we define the {\em renewal function}
\[
N_f(\phi,x,t):=\sum_{n=0}^{\infty}\sum_{y\in\sigma^{-n}(x)}\phi(y){\bf 1}_{\{S_{n}f(y)\leq t\}}(y).
\]

We recall that $N_{f}(\phi,x,t)$ satisfies the {\em renewal equation}
\begin{equation}
N_{f}(\phi,x,t)=\left(\sum_{y\in\sigma^{-1}(x)}N_{f}(\phi,y,t-f(y))\right)+\phi(x){\bf 1}_{\{t\geq0\}}(t)
\label{eq:RN eq}
\end{equation}

\begin{thm}
{\rm (Renewal theorem; Kesseb\"ohmer-Kombrink  \cite[Theorem 3.1]{kess-kom})}\label{thm:Renewal}
Suppose that $\Sigma^+$ is a topologically mixing, one-sided, countable Markov shift
with (BIP) and  $f:\Sigma^+\to\mathbb R$ is a strictly positive, non-arithmetic, locally H\"older continuous function
so that there exists a unique $\delta>0$ so that $P(-\delta f)=0$ and
$\int_{\Sigma^+} t f\ d\mu_{-\delta f}<+\infty$ for all $t$ in some neighborhood of $\delta$,
where $\mu_{-\delta f}$
is an equilibrium state for $-\delta f$.

If $\phi:\Sigma^+\to\mathbb R$ is non-negative, bounded, not identically zero, 
and locally H\"older continuous and  there exists $c>0$ such that 
\[
N_f(\phi,x,t)\leq ce^{t\delta},
\]
then
\[
N_f(\phi,x,t)\sim\frac{e^{t\delta}}{\delta}h_{-\delta
f}(x)\frac{\int_{\Sigma^{+}}\phi\ d\nu_{-\delta f}}{\int_{\Sigma^{+}}f\
d\mu_{-\delta f}}
\]
as $t\to\infty$, uniformly for $x\in\Sigma^{+}$, where $h_{-\delta f}:\Sigma^+\to\mathbb R$ is a bounded strictly positive
function so that $\mathcal L_{-\delta f} h_{-\delta f}=h_{-\delta f}$,  $\nu_{-\delta f}$ is a probability measure on $\Sigma^+$ 
so that  $\mathcal L^*_{-\delta f} \nu_{-\delta f}=\nu_{-\delta f}$ and \hbox{$\mu_{-\delta f}=h_{-\delta f}\nu_{-\delta f}$}.
\end{thm}

\begin{rem}
  \label{rem:apply_kess_komb}
The Renewal Theorem we state above is a special case of \cite[Theorem 3.1 (i)]{kess-kom}.
Following the notations in \cite{kess-kom}, in our case
$\eta=0$ and $f_{y}(t)=\begin{cases}
1 & t\geq0\\
0 & \mathrm{otherwise}
\end{cases}$.
Kesseb\"ohmer and Kombrink \cite{kess-kom} in place of our assumption of non-arithmeticity only require the weaker assumption that 
$f$ is not a lattice, i.e. that $f$ is not cohomologous to a function so that $\{S_nf(x)\ |\ x\in\Sigma^+\}$ does
not lie in a discrete subgroup of $\mathbb R$.
Moreover, since $f_{y}(t)\geq0$,
$\int_{-\infty}^{\infty}e^{-T\delta}f_{y}(T)\ dT=\frac{1}{\delta}$, and
$N_f(\phi,x,t)=0$ for $t<0$ when $f$ is strictly positive,
their conditions (B) and (D) are satisfied. So, it  only remains to check that their condition (C)  is satisfied, which translates to the existence of $c>0$ such that
\[
N_{f}(\phi,x,t)\leq ce^{t\delta}.
\]
\end{rem}

\medskip

We first check that a weak entropy gap at infinity implies such a bound on $N_f({\bf 1}, x,t)$. 

\begin{lem}
\label{lem:est_for_RT}
Suppose that $\Sigma^+$ is a topologically mixing, one-sided, countable Markov shift
with (BIP) and  $f:\Sigma^+\to\mathbb R$ is a strictly positive, locally H\"older continuous function
with a weak entropy gap at infinity. Let $\delta>d(f)$ be the unique constant such that $P(-\delta f)=0$. Then
there exists $C>0$ such that
\[
N_f({\bf 1}, x,t)=\sum_{n=0}^{\infty}\sum_{y\in\sigma^{-n}(x)}\mathbf 1_{\{S_{n}f(y)\leq t\}}(y)\leq Ce^{t\delta }
\]
for all $x\in\Sigma^{+}$ and $t>0$.
\end{lem}

We adopt the strategy of Lalley \cite[Lemma 8.1]{lalley}.

\begin{proof}
Define 
for all $x\in\Sigma^+$ and $t>0$
$$G(x,t)=e^{-t\delta} \frac{N_f({\bf 1},x,t)}{h_{-\delta f}(x)}$$
where $h_{-\delta f}$ is the eigenfunction for the transfer operator given by Theorem \ref{transfer fact}.
Let
$$\widehat G(t)=\sup\{ G(x,s)\ | \ x\in\Sigma^+,\ s\le t\}.$$
Notice that $\widehat G(t)$ is finite for all $t>0$, since $h_{-\delta f}$ is bounded away from 0,
and for any fixed $t>0$ there exists only finitely many $a\in\mathcal A$
so that $I(f,a)\le t$ (which implies that there are only finitely many $n$ and only finitely many $y\in\sigma^{-n}(x)$, for each $n$, so that
$S_nf(y)\le t$).
Since $h_{-\delta f}$ is bounded away from $0$ and $\infty$,
it remains to show that there exists $\hat C$ so that 
\[
\widehat G(t)\leq\hat C
\]
for all $t>0$.

The renewal equation (\ref{eq:RN eq}) implies that
\begin{align*}
	G(x,t)= & \sum_{y:\ \sigma(y)=x}G(y,t-f(y))e^{-\delta f(y)}\frac{h_{-\delta f}(y)}{h_{-\delta f}(x)}+\frac{e^{-t\delta}}{h_{-\delta f}(x)}.
\end{align*}
for all  $t>0$.
Notice that since $h_{-\delta f}(x)$ is the eigenfunction of $\mathcal{L}_{-\delta f}$
with eigenvalue $1=e^{P(-\delta f)}$,
\[
\sum_{y:\sigma(y)=x}e^{-\delta f(y)}\frac{h_{-\delta f}(y)}{h_{-\delta f}(x)}=\frac{\left(\mathcal{L}_{-\delta f}h_{-\delta f}\right)(x)}{h_{-\delta f}(x)}=1.
\]

If $c=c(f)=\inf_{x\in\Sigma^{+}}f(x)>0$,
then
	\begin{equation}
		G(x,t)\leq\widehat{G}(t-c)+\frac{e^{-t\delta}}{h_{-\delta f}(x)}
	\end{equation}
for all $x\in\Sigma^+$ and $t\ge c$.
Therefore, 
	\[
	\widehat{G}(mc)\leq\widehat{G}(c)+\hat{H}\sum_{n=1}^{m}e^{-cn\delta}
	\]
for all $m\in\mathbb N$, where
$$\hat H=\sup\Big\{\frac{1}{h_{-\delta f}(x)}\ |\ x\in\Sigma^+\Big\}.$$
Since $\widehat G$ is increasing,
	\[
	\widehat{G}(t)\leq\hat{C}=\widehat{G}(c)+\hat{H}\sum_{n=1}^{\infty}e^{-cn\delta}
	\]
for all $t>0$, which completes the proof. 
\end{proof}
 
 If $\phi:\Sigma^+\to\mathbb R$ is bounded, non-negative and locally H\"older continuous, then 
 $$N_f(\phi,x,t)\le \Big( \sup_{x\in\Sigma^+} \phi(x)\Big) N_f({\bf 1}, x,t),$$
so  Lemmas \ref{gap and pressure}, \ref{equilibrium state exists}  and \ref{lem:est_for_RT} together  imply that we can apply the Renewal Theorem to $\phi$ when
$f$ is strictly positive and has a weak entropy gap at infinity.
 
 \begin{cor}
 \label{can apply renewal}
Suppose that $\Sigma^+$ is a topologically mixing, one-sided, countable Markov shift
with (BIP) and  $f:\Sigma^+\to\mathbb R$ is a strictly positive, non-arithmetic, locally H\"older continuous function with a weak entropy gap at infinity,
$P(-\delta f)=0$. If
$\phi:\Sigma^+\to\mathbb R$ is bounded, non-negative, 
not identically zero
and locally H\"older continuous, then
\[
N_f(\phi, x,t)\sim\frac{e^{t\delta}}{\delta}h_{-\delta f}(x)\frac{\int_{\Sigma^{+}}\phi\ d\nu_{-\delta f}}{\int_{\Sigma^{+}}f\ d\mu_{-\delta f}}
\]
as $t\to\infty$, uniformly for $x\in\Sigma^{+}$, where $h_{-\delta f}:\Sigma^+\to\mathbb R$ is a bounded strictly positive
function so that $\mathcal L_{-\delta f} h_{-\delta f}=h_{-\delta f}$,  $\nu_{-\delta f}$ is a probability measure on $\Sigma^+$ 
so that  $\mathcal L^*_{-\delta f} \nu_{-\delta f}=\nu_{-\delta f}$ and \hbox{$\mu_{-\delta f}=h_{-\delta f}\nu_{-\delta f}$} is the equilibrium state for $-\delta f$.
\end{cor}

\section{Preparing to count}
In this section we develop the technical tools needed in the proofs of our  counting result. The majority of these results bound
the size of various subsets of the shift space. Most importantly, we show that if $y\in\sigma^{-n}(x)$ and $S_nf(y)$ is ``large,''  then ``typically''
$S_nf(y)$ is close to $n\int_{\Sigma^+}f\ d\mu_{-\delta f}$.  These results and their proofs generalize Lalley \cite[Theorem 6]{lalley}. The fact
that our Markov shift is countable requires more delicate control of error estimates.

For each  cylinder $p$, we choose a sample point $z_p\in p$ which is not periodic. We then define
$$W(n,p,t)=\sum_{y\in\sigma^{-n}(z_p)} {\bf 1}_p(y) {\bf 1}_{\{ x\ |\ S_nf(x)\le t\}}(y) = \#\Big(p\cap\sigma^{-n}(z_p)\cap \{ x\ |\ S_nf(x)\le t\} \Big).$$ 
We show that the $W(n,p,t)$ may be used to approximate the size of $\mathcal M_f(n,t)$.
This allows us to replace the counting of fixed points with counting of pre-images of our sample points. 

If $k\in\mathbb N$, let $\Lambda_k$ be the countable partition of $\Sigma^+$ into $k$-cylinders. 

\begin{lem}
\label{lem:key}
Suppose that $\Sigma^+$ is a topologically mixing, one-sided countable Markov shift with (BIP), $f\colon\Sigma^+\to\mathbb R$ is locally H\"older continuous
strictly positive and has a weak entropy gap at infinity.  If $P(-\delta f)=0$ and $\mu_{-\delta f}$ is the equilibrium state for $-\delta f$, then

\begin{itemize}
\item[(i)] If $v_k=\inf\{\mu_{-\delta f}(p) \ |\ p\in\Lambda_k\}$, then
$\lim_{k\to\infty} v_k=0$. 

\item[(ii)] For any $p\in\Lambda_k$ and $n\ge k$ there exists a bijection
$$\Psi_p^n:\mathrm{Fix}^n\cap p\to \sigma^{-n}(z_p)\cap p.$$

\item[(iii)] There exists a sequence $\{\epsilon_k\}$ such that 
$\lim\epsilon_k=0$ and if $y\in\mathrm{Fix}^n\cap p$ and $n\ge k$, then
$$|S_nf(y)-S_nf(\Psi_p^n(y))|\le\epsilon_k.$$

\item[(iv)] 
If $n\ge k$, then
\begin{equation}
\sum_{p\in\Lambda_k} W(n,p,t-\epsilon_{k})\leq\#\mathcal M_f(n,t)\leq
\sum_{p\in\Lambda_k}W(n,p,t+\epsilon_{k}).\label{eq:keylemma2}
\end{equation}

Moreover,
for all $k\in\mathbb N$ and $s\in(d(f),\delta)$, there exists $C(k,s)>0$ such that for any
$n <k$ and $t>0$, 
\[
\sum_{p\in\Lambda_{k}}W(n,p,t)\leq C(k,s)e^{st}\qquad\mathrm{and}\qquad \#\cal M_f(k,t)\le C(k,s)e^{st}.
\]
\end{itemize}

\end{lem}

\begin{proof}
Recall that since $\mu_{-\delta f}$ is a Gibbs state for $-\delta f$ (see Theorem \ref{Gibbs is eq}) and $P(-\delta f)=0$,
there exists $B>1$ such that for every $p\in\Lambda_k$, and $x\in p$
\[
\mu_{-\delta f}(p)\leq B e^{-\delta S_kf(x)}.
\] 
Since $f$ is strictly positive, $\lim_{k\to\infty}\inf\{ S_k(x)\ |\ x\in\Sigma^+\}=+\infty$, so (i) holds.
 
\medskip

Given $p\in\Lambda_k$, we define an explicit bijection
\[
\Psi_p^n:\mathrm{Fix}^{n}\cap p\to\sigma^{-n}(z_p)\cap p
\]
If $y=\overline{y_{1}y_{2}...y_{n}}\in\mathrm{Fix}^{n}\cap p$, then let
$$\Psi^n_p(y)=y_1\cdots y_n z_1\cdots z_m\cdots.$$
Notice that since $y_1=z_1$ and $\overline{y_1\cdots y_n}\in\Sigma^+$, we 
must have $t_{y_n y_1}=t_{y_n z_1}=1$, so $\Psi^n_p(y)\in\Sigma^+$.
The map $\Psi^n_p$ is injective by definition. If $x\in\sigma^{-n}(z_p)\cap p$, then, since $n\ge k$, $x_{n+1}=z_1=x_1$, which implies that
$\overline{x_1\cdots x_n}\in \mathrm{Fix}^n\cap p$, so $\Psi^n_p$ is also surjective.
Thus, we have established (ii).

\medskip

Since $f$ is locally H\"older continuous, there exists
$B>0$ and $r\in(0,1)$ so that 
\[
|f(x)-f(y)|\leq Br^{l}
\]
if $x_{i}=y_{i}$ for all $i\le l$. 
Therefore, if $y\in \text{Fix}^n\cap p$, then, since
 $z_p\in p$,
$y_i=\Psi^n_p(y)_i$ for all $i\le n+k$, so
\[
  |S_{n}f(y)-S_{n}f(\Psi^n_p(y))|\leq\epsilon_{k}=B\sum_{l=k}^{\infty}r^{l}. 
\]

The first statement in (iv) follows immediately from (ii) and (iii).
Choose $b\in (d(f),z)$. 
Lemma \ref{growth of alphabet} implies that  there exists $D$ so that 
$$B_1(f,t)=\#\big\{a\in\mathcal A\ |\ I(f,a) \le t\big\}\le  De^{bt}.$$
If
$$c=c(f)=\inf_{x\in\Sigma^+} f(x)=\inf_{a\in\mathcal A} I(f,a)>0$$
and $r\in\mathbb N$,
then
\begin{align*}
B_2(f,rc)&=
\#\big\{(a_1,a_2)\in \mathcal A\times\mathcal A\ |\ I(f,a_1)+I(f,a_2)\le rc\big\}\\
&\le\sum_{s=1}^r B_1(f,rc-sc)B_1(f,sc)
\le 
\sum_{s=1}^r D^2 e^{brc}= rD^2e^{brc}. 
\end{align*}
We may use the  argument above to inductively show that 
$$B_k(f,rc)  = \#\Big\{(a_i)\in \mathcal A^k\ \Big|\ \sum_{i=1}^k I(f,a_i)
\le rc\Big\} \le r^{k-1} D^ke^{brc}.$$
Notice that
$$\sum_{p\in\Lambda_{k}}W(n,p,rc)\leq B_n(f,rc) 
\qquad\mathrm{and}\qquad  \#\cal M_f(k,rc)\le B_k(f,rc)$$
so (iv) follows. 
\end{proof}

We set up some convenient notation. If $x\in\Sigma^+$, let
$$\mathcal W(x,t)=\left\{y\in\Sigma^{+}\ \Big|\ \sigma^{n}(y)=x,\ S_{n}f(y)\leq t \ \mathrm{for\ some}\ n\geq1\right\}$$
Observe that if  $x$ is not periodic and $y\in \mathcal W(x,t)$, then
there is a unique $n(y)$ so that $\sigma^{n(y)}(y)=x$.  If $x$ is not periodic and $\epsilon>0$, we let
$$\mathcal W(x,t,\le\epsilon)=\Big\{y\in\mathcal W(x,t)\ \big| \ \Big|\frac{t}{n(y)}-\bar f\Big|\le\epsilon\Big\}, \qquad\mathrm{and}$$
$$\mathcal W(x,t,>\epsilon)=\Big\{y\in\mathcal W(x,t)\ \big| \ \Big|\frac{t}{n(y)}-\bar f\Big|>\epsilon\Big\}=\mathcal W(x,t)-\mathcal W(x,t,\le\epsilon)$$
where $\bar f=\int_{\Sigma^+}f\ d\mu_{-\delta f}$.
Moreover, let
$$W(x,t)=\#\mathcal W(x,t),\ \  W(x,t,<\epsilon)=\#\mathcal W(x,t,\le\epsilon)\ \ \mathrm{and}\ \ W(x,t,>\epsilon)=\#\mathcal W(x,t,>\epsilon)=W(x,t)-W(x,t,\le\epsilon).$$ 

The crucial technical result we need for the proof of our counting result is a uniform bound on the growth of $W(x,t,>\epsilon)$.

\begin{prop}
\label{growth of bad set}
Suppose that $\Sigma^+$ is a topologically mixing, one-sided, countable Markov shift
with (BIP) and  $f:\Sigma^+\to\mathbb R$ is a strictly positive, locally H\"older continuous function
with a weak entropy gap at infinity.  Let $\delta>d(f)$ be the unique constant such that $P(-\delta f)=0$. 
Given $\epsilon>0$, there exist $D>0$  and $b<\delta$ so that
$$ W(x,t,>\epsilon) \le De^{bt}$$
for any non-periodic $x\in\Sigma^+$.
\end{prop}

\begin{proof}
Fix, for the entire proof, $\epsilon\in(0,\bar f/2)$.

Theorem \ref{bounds on transfer} implies that if  $s>d(f)$,
then there exist $R_s>0$ and $\eta_s\in (0,1)$ so that
\begin{equation}
\Big\| e^{-nP(-sf)}\mathcal{L}_{-sf}^{n}\mathbf{1}(x)-h_{-sf}(x)\int \mathbf{1} d\nu_{-zf}\Big\|\le R_{s}\eta_{s}^{n}.\label{MUbounds}
\end{equation}

If $s>\delta$, then $P(-sf)<0$, since $P(-\delta f)=0$ and $s\to P(-sf)$ is monotone decreasing and continuous on $(d(f),\infty)$
(by Lemma \ref{gap and pressure}).
Then, for any $m\in\mathbb N$ and  $t>0$
\begin{eqnarray*}
		\sum_{n\ge m}\ \ \sum_{y\in\sigma^{-n}(x)} {\bf 1}_{\{S_{n}f(y)\le t\} }(y) & \leq & \sum_{n\ge m}\ \sum_{y\in\sigma^{-n}(x)} e^{-s\left(S_{n}f(y) -t\right)} \\
		& = & e^{st} \sum_{n\geq m}\left(\mathcal{L}_{-sf}^{n}\mathbf{1}\right)(x)\\
		& \leq & e^{st} \sum_{n\ge m}e^{nP(-sf)}\left(h_{-sf}(x)+R_{s}\eta_{s}^{n}\right)\\
		& \leq & e^{st}\Big(\frac{e^{mP(-sf)}}{1-e^{P(-sf)}}\big(H_{s}+R_{s}\big)\Big).
\end{eqnarray*}
where $H_{s}=\sup \{ h_{-sf}(x)\ |\ x\in\Sigma^+\}$.

If  $ \frac{t}{n(y)}-\bar f<-\epsilon$, then $n(y)\bar f>t+n(y)\epsilon$ and $n(y)>\frac{t}{\bar f-\epsilon}$, so $n(y)\bar f>t(1+\epsilon_1)$
where $\epsilon_1=\frac{\epsilon}{\bar f-\epsilon}$.
Given $t>0$, let $m_t=\left\lfloor \frac{t(1+\epsilon_1)}{\bar f}\right\rfloor$. Then
\begin{eqnarray*}
\#\big\{y\in\mathcal W(x,t)\ |\  \frac{t}{n(y)}-\bar f<-\epsilon\big\} &\le & \sum_{n\ge m_t}\ \ \sum_{y\in\sigma^{-n}(x)} {\bf 1}_{\{S_{n}f(y)\le t\} }(y)\\
&\le &  e^{st}\Big(\frac{e^{m_tP(-sf)}}{1-e^{P(-sf)}}\big(H_{s}+R_{s}\big)\Big).\\
& \le & D_0e^{st+m_tP(-sf)}
\end{eqnarray*}
where $D_0=D_0(s,f,\epsilon)=\frac{H_s+R_{s}}{1-e^{P(-sf)}}$.

Since $\left.\frac{d}{ds}\right|_{s=\delta}P(-sf)=-\overline{f}<0$
(by Theorem \ref{pressure analytic}), we may also choose $s>\delta$ so that 
\[b_0:=s+\frac{1+\epsilon_1}{\bar f} P(-sf)<\delta.\]
Notice that $b_0$  does depend on $\epsilon$.

With this choice of $s$,
$$\#\big\{y\in \mathcal W(x,t)\ |\  \frac{t}{n(y)}-\bar f<-\epsilon\big\} \le D_0e^{b_0t}.$$

One can similarly show that there exist $D_1>0$ and $b_1\in (d(f),\delta)$ so that
$$\#\big\{y\in \mathcal W(x,t)\ |\  \frac{t}{n(y)}-\bar f>\epsilon\big\} \le D_1e^{b_1t}.$$
(In this case, we choose $r\in (d(f),\delta)$ so that 
\[b_1:=r+\frac{1-\epsilon_2}{\bar f} P(-rf)<\delta\]
where $\epsilon_2=\frac{\epsilon}{\overline{f}+\epsilon}>0$. 
We then use Equation (\ref{MUbounds}) and an analysis similar to the one above to  show that
$$\#\big\{y\in \mathcal W(x,t)\ |\  \frac{t}{n(y)}-\bar f>\epsilon\big\} \le D_1e^{t\big(r+ \frac{1-\epsilon_2}{\bar f} P(-rf)\big)}$$
where $D_1=D_1(r,f,\epsilon)=e^{P(-rf)}(H_r+R_{r})$.)

So,
$$W(x,t,>\epsilon)\le D_0e^{b_0t}+D_1e^{b_1t}\le D e^{bt}$$
where $D=D_0+D_1$ and $b=\max\{b_1,b_2\}<\delta$.
\end{proof}

\begin{cor}
\label{probs}
Suppose that $\Sigma^+$ is a topologically mixing, one-sided, countable Markov shift
with (BIP) and  $f:\Sigma^+\to\mathbb R$ is a strictly positive, locally H\"older continuous function
with a weak entropy gap at infinity.  Let $\delta>d(f)$ be the unique constant such that $P(-\delta f)=0$. 
Then, given any $\epsilon>0$, there exists $a>0$ so that 
\begin{enumerate}
\item
There exists $\hat D>0$ so that
$$\frac{W(x,t,>\epsilon)}{W(x,t)}\le \hat D e^{-at}$$
for any non-periodic  $x\in\Sigma^+$.
\item 
Given any cylinder $p$, there exists $D_p$ so that
$$\frac{\#(\mathcal W(x,t,>\epsilon)\cap p)}{\#(\mathcal W(x,t)\cap p)}\le D_pe^{-at}$$
for any non-periodic $x\in\Sigma^+$.
\end{enumerate}
\end{cor}

\begin{proof}
By Corollary \ref{can apply renewal} we can apply the Renewal Theorem with $\phi={\bf 1}$ to see that 
\begin{equation}
		N_{f}(\mathbf{1},x,t)=W(x,t)+1=\sum_{n\geq0}\sum_{\sigma^{n}(y)=x}  {\bf 1}_{\{S_{n}f(y)\leq t \}}(y)\sim\frac{h_{-\delta f}(x)}{\delta\bar f} e^{t\delta}
\end{equation}
uniformly in $x\in\Sigma^+$, where $\sim$ indicates that the ratio goes to 1 as $t\to\infty$.
Since there exist $b<\delta$ and $D>0$ so that $W(x,t,>\epsilon)\le De^{bt}$, (1) holds with $a=\delta-b$ and some $\hat D>0$.

We can similarly apply the Renewal Theorem  with $\phi={\bf 1}_p$ to conclude that
$$N_{f}(\mathbf{1}_p,x,t)=\#(\mathcal W(x,t)\cap p)+1=\sum_{n\geq0}\sum_{\sigma^{n}(y)=x} {\bf 1}_p {\bf 1}_{\{S_{n}f(y)\leq t \}}(y)\sim\frac{\nu(p)h_{-\delta f}(x)}{\delta\bar f} e^{t\delta}$$
uniformly in $x\in\Sigma^+$.
Since $\nu(p)>0$ and 
$$\#(\mathcal W(x,t,>\epsilon)\cap p)\le W(x,t,>\epsilon)\le  De^{bt},$$
(2) holds for some $D_p$ depending on the cylinder $p$.
\end{proof}

The following result will allow us to bound the error terms in our approximations. Given $T>0$, let
$$P^k_T=\{p\in\Lambda_k \ |\ S_kf(z_p)\le T\}\qquad\mathrm{and}\qquad Q_T^k=\Lambda_k-P_T^k.$$
Notice that $P_T^k$ is finite for all $k$ and $T$.

\begin{cor}
\label{coarse estimate}
Suppose that $\Sigma^+$ is a topologically mixing, one-sided, countable Markov shift
with (BIP) and  $f:\Sigma^+\to\mathbb R$ is a strictly positive, locally H\"older continuous function
with a weak entropy gap at infinity.  Let $\delta>d(f)$ be the unique constant such that $P(-\delta f)=0$. 
\begin{enumerate}
\item
There exists $G>0$ so that
$$\sum_{n\ge 1}\sum_{\{y\in \sigma^{-n}(x)\}}  \frac{1}{n} {\bf 1}_{ \{S_nf(y)\le t\}}(y)\le G\frac{e^{t\delta}}{t}$$
for any $x\in\Sigma^+$ and all $t>0$.
\item
If $k\in\mathbb N$ and $t>T>0$, then
$$\sum_{n>k}\sum_{\{y\in\sigma^{-n}(x)\}}  \frac{1}{n}{\bf 1}_{Q_T^k}(y) {\bf 1}_{\{S_nf(y)\le t\}} (y)\le Ge^{-T\delta}\frac{e^{t\delta}}{t-T}.$$
\end{enumerate}
\end{cor}

\begin{proof}
Fix some $\epsilon>0$. Recall from Lemma \ref{lem:est_for_RT} that $W(x,t)\le Ce^{t\delta}$ for all $x\in\Sigma^+$.
Then
\begin{eqnarray*}
\sum_{n\ge 1}\sum_{y\in\sigma^{-n}(x)} \frac{1}{n} {\bf 1}_{\{S_nf(y)\le t\} } (y) & = &\sum_{y\in \mathcal W(x,t,\le\epsilon)}\frac{1}{n(y)} +
\sum_{y\in \mathcal W(x,t,>\epsilon)}\frac{1}{n(y)}\\
 & \leq &\sum_{y\in \mathcal W(x,t,\leq\epsilon)}\left(\frac{\bar f+\epsilon}{t}\right){\bf 1}(y)+
 \sum_{y\in \mathcal W(x,t,>\epsilon)}{\bf 1}(y).\\
 & \le &  Ce^{t\delta} \left(\frac{\bar f+\epsilon}{t}\right) +\Big(\hat D e^{-at} \Big) Ce^{t\delta}.
\end{eqnarray*}
So, (1) holds for some $G>0$.

\medskip

Now notice that
\begin{eqnarray*}
\sum_{n>k}\sum_{y\in\sigma^{-n}(x)} \frac{1}{n} {\bf 1}_{Q_T^k}(y){\bf 1}_{\{S_nf(y) \le t\}}(y) & \le  & \sum_{n  >  k}\frac{1}{n} \sum_{y\in\sigma^{k-n}(x)} {\bf 1}_{\{S_{n-k}f(y)\le t -T\}}(y)\\
&= & \sum_{m\ge 1} \sum_{w\in\sigma^{-m}(x) }\frac{1}{m+k}{\bf 1}_{\{S_mf(w)\le t -T\} }(w)\\
&\le & \sum_{m\ge 1} \sum_{w\in\sigma^{-m}(x) }\frac{1}{m}{\bf 1}_{\{S_mf(w)\le t -T\} }(w)\\
& \le & Ge^{-\delta T}\frac{e^{t\delta}}{ t-T}
\end{eqnarray*} 
which completes the proof of (2).
\end{proof}

\section{Counting} 
\begin{proof}[Proof of Theorem \ref{thm:countingN}]
First notice that Lemma \ref{eventually to strictly} implies that we may assume that $f$ is strictly positive
and has a weak entropy gap at infinity. 

We simplify notation by setting $\mu=\mu_{-\delta f}$, $\nu=\nu_{-\delta f}$, $h=h_{-\delta f}$, and $\bar f=\int f\ d\mu$, 
where $h_{-\delta f}:\Sigma^+\to\mathbb R$ is a bounded strictly positive
function so that $\mathcal L_{-\delta f} h_{-\delta f}=h_{-\delta f}$,  $\nu_{-\delta f}$ is a probability measure on $\Sigma^+$ 
so that  $\mathcal L^*_{-\delta f} \nu_{-\delta f}=\nu_{-\delta f}$ and \hbox{$\mu_{-\delta f}=h_{-\delta f}\nu_{-\delta f}$} is the equilibrium state for $-\delta f$.

Suppose that $p\in\Lambda_k$. Corollary \ref{can apply renewal} implies that
we can apply the Renewal Theorem (Theorem \ref{thm:Renewal}) with $\phi={\bf 1}_p$. Therefore,
\begin{alignat*}{1}
L(p,t):=\#(\mathcal W(z_p,t)\cap p)
 & =\sum_{n\geq 1}\sum_{y\in\sigma^{-n}(z_p)}{\bf 1}_{p}(y) {\bf 1}_{\{S_{n}f(y)\leq t\}}(y) \sim C(p)e^{t\delta}
\end{alignat*}
where
\[
C(p)=\frac{h(z_p)\nu(p)}{\delta\bar f}.
\]

Fix, for the moment, $p\in\Lambda_k$.
We define
\[
\widehat L(p,t):= \sum_{n\geq 1}\frac{1}{n} W(n,p,t)=\sum_{y\in \mathcal W(z_p,t)}\frac{1}{n(y)} {\bf 1}_p(y).
\]
Then
\begin{alignat*}{1}
 \widehat {L}(p,t) & =\sum_{y\in \mathcal W(z_p,t,\le\epsilon)}\frac{1}{n(y)} {\bf 1}_p(y)+
\sum_{y\in \mathcal W(z_p,t,>\epsilon)}\frac{1}{n(y)}{\bf 1}_p(y)\\
 & \leq\sum_{y\in \mathcal W(z_p,t,\leq\epsilon)}\left(\frac{\bar f+\epsilon}{t}\right){\bf 1}_p(y)+
 \sum_{y\in \mathcal {W} (z_p,t,>\epsilon)}{\bf 1}_p(y).
\end{alignat*}
Since, by Corollary \ref{probs},
$$\#\Big( \mathcal W(z_p,t,>\epsilon)\cap p\Big)\le D_pe^{-at}\# \Big(\mathcal W(z_p,t)\cap p\Big)$$
for some $D_p,a>0$, it follows that
$$\limsup_{t\to \infty} \frac{t\widehat L(p,t) }{L(p,t)}\le \bar f+\epsilon.$$
Similarly, 
$$
\widehat L(p,t)=\sum_{n\geq 1}\frac{1}{n}W(n,p,t) \geq\sum_{y\in \mathcal W(z_p,t,\leq\epsilon)}\left(\frac{\bar f-\epsilon}{t}\right){\bf 1}_p(y)$$
so
$$\liminf_{t\to \infty} \frac{t\widehat L(p,t)}{L(p,t)}\ge \bar f-\epsilon.$$

By letting $\epsilon\to0,$ we  see that
$$\widehat L(p,t)\sim \frac{\bar f L(p,t)}{t}\sim \frac{C(p)\bar f}{t} e^{t\delta}.$$

Now suppose that $P$ is a subset of $\Lambda_k$ and define
$$L(P,t)=\sum_{p\in P} L(p,t)\qquad\mathrm{and}\qquad\widehat L(P,t)=\sum_{p\in P} \widehat L(p,t).$$
The above analysis implies that if $P$ is {\bf finite}, then
$$L(P,t)\sim \sum_{p\in P}  C(p) e^{t\delta}\qquad\mathrm{and}\qquad \widehat L(P,t)\sim \sum_{p\in P}  \frac{C(p)\bar f}{t} e^{t\delta}.$$
Notice that if $T>0$ and $t>T$, then Corollary \ref{coarse estimate} and Lemma \ref{lem:key} imply that there exists $C_k>0$ so that
$$\frac{t\widehat L(P^k_T,t)}{e^{t\delta}}\le\frac{t  \widehat L(\Lambda_k,t)}{e^{t\delta}}\le \frac{t \widehat L(P^k_T,t)}{e^{t\delta}} + tC_k e^{(s-\delta)t}+
Ge^{-\delta T}\frac{t}{ t-T}$$
for some $s\in (d(f),\delta)$, so
$$\bar f\sum_{p\in P^k_T} C(p)\le \liminf_{t\to \infty} \frac{t\widehat L(\Lambda_k,t)}{e^{t\delta}}\le\limsup _{t\to \infty} \frac{t\widehat L(\Lambda_k,t)}{e^{t\delta}}\le\bar f\sum_{p\in P_T^k} C(p) + 
Ge^{-\delta T}$$
Applying the above inequality to the sequence $\{P_T^k\}_{T\in\mathbb N}$, we conclude that
$$ \widehat L(\Lambda_k,t)\sim \sum_{p\in \Lambda_k}  \frac{C(p)\bar f}{t} e^{t\delta}.$$

Lemma \ref{lem:key} implies that,  given $k\in\mathbb N$ there exists $s<\delta$ and $C_k>0$, so that
$$\sum_{p\in\Lambda_k}\sum_{n=1}^k\frac{1}{n}W(n,p,t)\le C_k e^{st}
\qquad\mathrm{and}\qquad \sum_{n=1}^k\frac{1}{n}\#\big(\mathcal M_f(n,t)\big) \le C_ke^{st}$$
and 
$$\sum_{p\in\Lambda_k}\sum_{n=k}^\infty\frac{1}{n}W(n,p,t-\epsilon_k)\le\sum_{n=k}^\infty \frac{1}{n}\#\big(\mathcal M_f(n,t)\big)\le
\sum_{p\in\Lambda_k}\sum_{n=k}^\infty\frac{1}{n}W(n,p,t+\epsilon_k).$$
Therefore,  recalling that $M_f(t)=\sum_{n\ge 1} \frac{1}{n}\#\big(\mathcal M_f(n,t)\big)$, we see that
$$\widehat L(\Lambda_k,t-\epsilon_k)-C_k e^{st}\le M_f(t)\le  \widehat L(\Lambda_k,t+\epsilon_k)+C_k e^{st},$$
so
$$e^{-\delta\epsilon_k}\bar f\sum_{p\in\Lambda_k} C(p)\le \liminf _{t\to \infty} \frac{t M_f(t)}{e^{t\delta}}\le\limsup_{t\to \infty} \frac{t M_f(t)}{e^{t\delta}}\le e^{\delta\epsilon_k}\bar f\sum_{p\in\Lambda_k} C(p)$$

Since $h$ is bounded and continuous and $v_k=\sup\{ \mu(p)\ |\ p\in\Lambda_k\}\to 0$ as $k\to\infty$, by Lemma \ref{lem:key} (i),
\[
\sum_{p\in\Lambda_k}C(p)=\frac{1}{\delta\bar f}\sum_{p\in\Lambda_k}h(z_p)\nu(p)\to\frac{\int h\ d\nu}{\delta\bar f}=\frac{1}{\delta\bar f}.
\]
as $k\to\infty$. Moreover, $\lim \epsilon_k=0$. So, finally, we may conclude that
$$M_f(t)\sim \frac{e^{t\delta}}{t\delta}$$
as desired.
\end{proof}

\section{Equidistribution}

We are almost ready to prove our equidistribution result, but first we must develop one
more bound in the spirit of \cite[Theorem 6]{lalley}.

\subsection{Preparing to equidistribute}

Suppose that $f:\Sigma^+\to\mathbb R$ and $g:\Sigma^+\to\mathbb R$ are both strictly positive, $f$ has a weak entropy gap at infinity
and $P(-\delta f)=0$. We simplify notation, throughout the section, by letting $\mu=\mu_{-\delta f}$ denote the equilibrium state of $-\delta f$ and setting
$\overline{f}:=\int f\ d\mu$ and $\overline{g}:=\int g\ d\mu$. Since $f$ and $g$ are strictly positive,
$$c(f)=\inf\{f(x)\ |\ x\in\Sigma^{+}\}>0\qquad\mathrm{and} \qquad c(g)=\inf\{g(x)\ |\ x\in\Sigma^{+}\}>0.$$

\begin{prop}
\label{count for g}
Suppose that $\Sigma^+$ is a topologically mixing, one-sided, countable Markov shift
with (BIP) and  $f:\Sigma^+\to\mathbb R$ is a strictly positive, locally H\"older continuous function
with a weak entropy gap at infinity.  Let $\delta>d(f)$ be the unique constant such that $P(-\delta f)=0$. 
Further suppose that $g:\Sigma^+\to\mathbb R$ is strictly positive and that there exists $C>0$ so
that $|f(x)-g(x)|\le C$ for all $x\in\Sigma^+$.
Given $\epsilon>0$, there exist $A>0$  and $a<\delta$ so that
$$\#\Big\{ y\in \mathcal W(x,t)\ :\ \Big|\frac{S_{n}g(y)}{n(y)}-\bar g\Big|>\epsilon, \Big|\frac{t}{n(y)}-\overline{f}\Big|\le\epsilon \ \Big\}\le Ae^{at}$$
for any non-periodic   $x\in\Sigma^+$.
\end{prop}

\begin{proof}
Fix $\epsilon>0$. We may assume that $\epsilon <\min\{ c(f), c(g)\}$.

If $\frac{S_{n(y)}g(y)}{n(y)}-\overline{g}<-\epsilon$, then $S_{n(y)}g(y)<n(y)\overline{g}-n(y)\epsilon.$ If, in addition,
$\big|\frac{t}{n(y)}-\overline{f}\big|\le\epsilon$, then \hbox{$t\le n(y)(\overline{f}+\epsilon)$}, so 
$$S_{n(y)}g(y)<n(y)\overline{g}-n(y)\epsilon\le n(y)\overline{g}-n(y)\frac{\epsilon}{2}-\frac{t\epsilon}{2(\overline{f}+\epsilon)}\le n(y)(\overline{g}-\epsilon_{3})-t\epsilon_{3}$$
where $\epsilon_{3}=\max\{\frac{\epsilon}{2},\frac{\epsilon}{2(\overline{f}+\epsilon)}\}>0$.

Proposition \ref{adding g in} implies that $s\to P(-sg-\delta f)$ is monotone decreasing and well-defined on $(d(f)-\delta,\infty)$. So, if $s>0$,
then $P(-sg-\delta f)<0$. 
Moreover, there exist an equilibrium state $\mu_{-sg-\delta f}$ for $-sg-\delta f$ and an eigenfunction $h_{-sg-\delta f}$
for $\mathcal{L}_{-sg-\delta f}$ with eigenvalue $e^{P(-sg-\delta f)}<1$.
Furthermore, since $\left.\frac{d}{ds}\right|_{s=0}P(-sg-\delta f)=-\overline{g}<0$
(by Theorem \ref{pressure analytic}) we may choose $s>0$ so that
$$-d_0:=s(\overline{g}-\epsilon_3)+P(-sg-\delta f)<0.$$

Theorem \ref{bounds on transfer} implies that there exist $\bar{R}_s>0$ and $\bar{\eta}_{s}\in(0,1)$ so that 
\begin{equation}
		\Big\| e^{-nP(-sg-\delta f)}\mathcal{L}_{-sg-\delta f}\mathbf{1}-h_{-sg-\delta f}(x)\int\mathbf{1}d\nu_{-sg-\delta f}\Big\|\le R_{s}\bar{\eta}_{s}^{n}\label{MUbounds2}
\end{equation}
for all $n\in\mathbb{N}$. Therefore, 
{\footnotesize{
\begin{eqnarray*}
\#\Big\{ y\in \mathcal W(x,t) : \frac{S_{n}g(y)}{n(y)}-\bar g<-\epsilon, \big|\frac{t}{n(y)}-\overline{f}\big|\le\epsilon \Big\}&\le&
\sum_{n\geq0}\sum_{\sigma^{n}(y)=x} \mathbf{1}_{\{y\ |\ S_{n}g(y)\leq n\cdot(\overline{g}-\epsilon_3)-t\epsilon_3,\ S_{n}f(y)\leq t\}}(y)\\
&\leq & \sum_{n\geq0}\sum_{\sigma^{n}(y)=x} e^{-s\big(S_{n}g(y)-n(\overline{g}-\epsilon_3)+t\epsilon_3\big)-\delta\big(S_{n}f(y)-t\big)}\\
&= & e^{t\delta-st\epsilon_3}\sum_{n\geq0}e^{n(s(\bar g-\epsilon_3)+P(-sg-\delta f))}\left(e^{-nP(-sg-\delta f)}\mathcal{L}_{-sg-\delta f}^{n}{\bf 1}\right)\\
&\leq & e^{t\delta-st\epsilon_3}\sum_{n\geq0}\left(h_{-sg-\delta f}(x)+\bar{R}_{s}\bar{\eta}_{s}^{n}\right)e^{-nd_0}\\
&\leq & D_0e^{t\delta-st\epsilon_3}
\end{eqnarray*}
}}
for all $x\in\Sigma^{+}$, and some $D_0>0$ (which depends on $\epsilon$, $s$, $g$ and $f$).

One may similarly show that there exist  $\epsilon_4>0$, $r<0$ and $D_1>0$ so that
$$\#\Big\{ y\in \mathcal W(x,t) :\ \frac{S_{n}g(y)}{n(y)}-\bar g >\epsilon, \big|\frac{t}{n(y)}-\overline{f}\big|\le\epsilon \ \Big\}\le D_1e^{t\delta+rt\epsilon_4}.$$
Therefore, our result holds with $A=D_0+D_1$ and $a=\max\{\delta-s\epsilon_3,\delta+r\epsilon_4\}$.
\end{proof}

\subsection{Proof of Theorem \ref{thm:equid_roof}}
Lemma \ref{eventually to strictly} again
implies that we may assume that $f$ and $g$ are strictly positive and $f$ has a weak entropy gap
at infinity. Recall, from Lemma \ref{lem:key}, that there exists a
sequence $\{\epsilon_{k}\}$ so that $\lim\epsilon_{k}=0$, and, for
any $p\in\Lambda_{k}$ and $n\ge k$, there exists a bijection 
\[
\Psi_{p}^{n}:\mathrm{Fix}^{n}\cap p\to\sigma^{-n}(z_{p})\cap p
\]
so that 
\[
|S_{n}f(x)-S_{n}f(\Psi_{p}^{n}(x))|\le\epsilon_{k}\qquad\mathrm{and}\quad|S_{n}g(x)-S_{n}g(\Psi_{p}^{n}(x))|\le\epsilon_{k}
\]
for all $x\in\mathrm{Fix}^{n}\cap p$. 
Since $\lim\epsilon_{k}=0$, there exists $k_{0}$ so that if $n\ge k\ge k_{0}$,
then 
$$c=\min\{c(f),c(g)\}>2\epsilon_{k}.$$

We assume from now on that $k\ge k_{0}$. 
Then, if  $p\in\Lambda_k$
\begin{equation}
	\sum_{n\ge k}\frac{1}{n}\sum_{\stackrel[S_{n}f(x)\leq t]{x\in\mathrm{Fix}^{n}\cap p}{}}\frac{S_{n}g(x)}{S_{n}f(x)}\leq\sum_{y\in\mathcal W(z_p,t+\epsilon_k)\cap p}
	\frac{1}{n(y)}\left(\frac{S_{n}g(y)+\epsilon_{k}}{S_{n}f(y)-\epsilon_{k}}\right){\bf 1}_{\{n(y)\geq k\}}(y)\label{eq:upperbound}
\end{equation}
and 
\begin{equation}
	\sum_{n\geq k}\frac{1}{n}\sum_{\stackrel[S_{n}f(x)\leq t]{x\in\mathrm{Fix}^{n}\cap p}{}}\frac{S_{n}g(x)}{S_{n}f(x)}\geq\sum_{y\in\mathcal W(z_p,t-\epsilon_k)\cap p}\frac{1}{n(y)}
	\left(\frac{S_{n}g(y)-\epsilon_{k}}{S_{n}f(y)+\epsilon_{k}}\right){\bf 1}_{\{n(y)\geq k\}}(y).\label{eq:lowerbound}
\end{equation}

Since there exists $C>0$ so that $|f(x)-g(x)|\le C$ for
all $x\in\Sigma^{+}$, $S_{n(y)}f(y)\ge cn(y)$ for all $y\in \Sigma^+$ and $c>2\epsilon_k$, we see that
\[
\frac{S_{n}g(y)}{S_{n}f(y)}\le\frac{nC+S_{n}f(y)}{S_{n}f(y)}\le\hat{C}=\frac{C}{c}+1\ \ \ \mathrm{and}\ \ \ \frac{S_{n(y)}g(y)+\epsilon_{k}}{S_{n(y)}f(y)-\epsilon_{k}}\le 3\hat C.
\]

Let 
$$\mathcal V(x,t,\le\epsilon)=\Big\{ y\in \mathcal W(x,t)\ :\ \Big|\frac{S_{n}f(y)}{n(y)}-\bar f\Big|\le\epsilon, \Big|\frac{S_{n}g(y)}{n(y)}-\bar g\Big|\le\epsilon \ \Big\}.$$
Given $\epsilon>0$ so that $2\epsilon+2\epsilon_k < \bar f$. 
Proposition \ref{growth of bad set} together with Proposition \ref{count for g}, applied to both $f$ and $g$, imply that there exist $\hat A>0$ and $\hat a<\delta$ so that
$$\#\big(\mathcal W(x,t)\setminus \mathcal V(x,t,\le\epsilon)\big)\le \hat Ae^{\hat a t}$$
for all $t>0$.
Further recall that we saw in the proof of Theorem \ref{thm:countingN} that 
$$\widehat L(p,t)=\sum_{y\in \mathcal W(z_p,t+\epsilon_{k})\cap p}\frac{1}{n(y)} \sim C(p)\bar f\frac{e^{t\delta}}{t}.$$
Notice that
\begin{eqnarray*}
U(p,t+\epsilon_k) &:=&\sum_{y\in \mathcal W(z_p,t+\epsilon_{k})\cap p}\frac{1}{n(y)} \left(\frac{S_{n(y)}g(y)+\epsilon_{k}}{S_{n(y)}f(y)-\epsilon_{k}}\right)\\
&\leq&  
\left(\sum_{y\in \mathcal V(z_p,t+\epsilon_{k},\le\epsilon)\cap p}\frac{1}{n(y)}\left(\frac{\overline{g}+\epsilon+\frac{\epsilon_{k}}{n(y)}}{\overline{f}-\epsilon-\frac{\epsilon_{k}}{n(y)}}\right)\right)
	\ +3\hat C\#\Big( \mathcal W(z_p,t+\epsilon_{k})\setminus \mathcal V(z_p,t+\epsilon_k,\le\epsilon)\Big)\\
	& &\ \ \ \ \ +3\hat C\sum_{n=1}^{k-1} W(n,p,t)
\end{eqnarray*}
and recall, from Lemma \ref{lem:key}, that given $s\in (d(f),\delta)$, there exists $C(k,s)$ so that
$$W(n,p,t)\le C(k,s)e^{st}\qquad \mathrm{and}\qquad \#\mathcal M_f(t)\le C(k,s)e^{st}$$
for all $n<k$.
Therefore,
$$\limsup_{t\to\infty} \frac{U(p,t+\epsilon_k)}{\hat L(p,t+\epsilon_k)} \le\frac{\overline{g}+\epsilon+\epsilon_{k}}{\overline{f}-\epsilon-\epsilon_k}.$$
Letting  $\epsilon\to 0$, we see that
$$\limsup_{t\to\infty} \frac{U(p,t+\epsilon_k)}{\hat L(p,t+\epsilon_k)}\le \frac{\overline{g}+\epsilon_{k}}{\overline{f}-\epsilon_k}.$$

We can similarly show that if 
$$Z(p,t-\epsilon_k)=\sum_{y\in\mathcal W(z_p,t-\epsilon_k)}\frac{1}{n(y)}
	{\bf 1}_{p}(y)\left(\frac{S_{n}g(y)-\epsilon_{k}}{S_{n}f(y)+\epsilon_{k}}\right),$$
then
$$\liminf_{t\to\infty} \frac{Z(p,t-\epsilon_k)}{\hat L(p,t-\epsilon_k)}\ge \frac{\overline{g}-\epsilon_{k}}{\overline{f}+\epsilon_k}.$$
Therefore,
$$ \frac{\overline{g}-\epsilon_{k}}{\overline{f}+\epsilon_k}\le \liminf_{t\to\infty}\frac{1}{\widehat L(p,t-\epsilon_k)} \sum_{n\ge k}\frac{1}{n}\sum_{\stackrel[S_{n}f(x)\leq t]{x\in\mathrm{Fix}^{n}\cap p}{}}\frac{S_{n}g(x)}{S_{n}f(x)}
\le  \limsup_{t\to\infty}\frac{1}{\widehat L(p,t+\epsilon_k)} \sum_{n\ge k}\frac{1}{n}\sum_{\stackrel[S_{n}f(x)\leq t]{x\in\mathrm{Fix}^{n}\cap p}{}}\frac{S_{n}g(x)}{S_{n}f(x)}\leq
\frac{\overline{g}+\epsilon_{k}}{\overline{f}-\epsilon_k}.$$

Since $P_T^k$ is a finite set of cylinders, for any $T$ and $k$, we see that
$$ \frac{\overline{g}-\epsilon_{k}}{\overline{f}+\epsilon_k}\le \liminf_{t\to\infty}\frac{1}{\widehat L(P_T,t-\epsilon_k)} \sum_{n\ge k}\frac{1}{n}\sum_{\stackrel[S_{n}f(x)\leq t]{x\in\mathrm{Fix}^{n}\cap P_T^k}{}}\frac{S_{n}g(x)}{S_{n}f(x)}
\le  \limsup_{t\to\infty}\frac{1}{\widehat L(P^k_T,t+\epsilon_k)} \sum_{n\ge k}\frac{1}{n}\sum_{\stackrel[S_{n}f(x)\leq t]{x\in\mathrm{Fix}^{n}\cap P_T^k}{}}\frac{S_{n}g(x)}{S_{n}f(x)}\leq
\frac{\overline{g}+\epsilon_{k}}{\overline{f}-\epsilon_k}.$$

Now notice that if $t>T>0$, Corollary \ref{coarse estimate} implies that
\begin{equation}
	\sum_{n\ge k}\frac{1}{n}\sum_{\stackrel[S_{n}f(x)\leq t]{x\in\mathrm{Fix}^{n}\cap Q_T^k}{}}\frac{S_{n}g(x)}{S_{n}f(x)}\leq3\hat C\widehat L(Q_T^k,t)
	\le 3\hat C Ge^{-\delta T}\frac{e^{t\delta}}{t-T}
\end{equation}
Therefore, as in the proof of Theorem \ref{thm:countingN}, we conclude that 
$$ \frac{\overline{g}-\epsilon_{k}}{\overline{f}+\epsilon_k}\le \liminf_{t\to\infty}\frac{1}{\widehat L(\Lambda_k,t-\epsilon_k)} \sum_{n\ge k}\frac{1}{n}\sum_{\stackrel[S_{n}f(x)\leq t]{x\in\mathrm{Fix}^{n}}{}}\frac{S_{n}g(x)}{S_{n}f(x)}
\le  \limsup_{t\to\infty}\frac{1}{\widehat L(\Lambda_k,t+\epsilon_k)} \sum_{n\ge k}\frac{1}{n}\sum_{\stackrel[S_{n}f(x)\leq t]{x\in\mathrm{Fix}^{n}}{}}\frac{S_{n}g(x)}{S_{n}f(x)}\leq
\frac{\overline{g}+\epsilon_{k}}{\overline{f}-\epsilon_k}.$$

Recall that   $\lim\epsilon_k=0$, 
$$\widehat L(\Lambda_k,t-\epsilon_k)-C_k e^{st}\le M_f(t)\le  \widehat L(\Lambda_k,t+\epsilon_k)+C_k e^{st},$$
for all $t>0$, and that
\[
\lim_{t\to\infty}M_f(t)\frac{t\delta}{e^{\delta t}}=1,
\]
so we see that 
\[
\sum_{n=1}^{\infty}\frac{1}{n}\sum_{\stackrel[S_{n}f(x)\leq t]{x\in\mathrm{Fix}^{n}}{}}\frac{S_{n}g(x)}{S_{n}f(x)}\sim\frac{\bar{g}}{\bar{f}}\frac{e^{t\delta}}{t\delta}
\]
as desired. This completes the proof of Theorem \ref{thm:equid_roof}.

\section{The Manhattan curve}

Suppose that $f:\Sigma^{+}\to\mathbb{R}$ is locally H\"older continuous,
strictly positive and has a strong entropy gap
at infinity and that $g:\Sigma^{+}\to\mathbb{R}$ is also strictly
positive and locally H\"older continuous and there exists $C>0$ so
that $|f(x)-g(x)|<C$ for all $x\in\Sigma^{+}$. In this case, $c(f)=\inf\{f(x)\ |\ x\in\Sigma^+\}>0$
and $c(g)=\inf\{g(x)\ |\ x\in\Sigma^+\}>0$.

In this case we define, the enlarged {\em Manhattan curve}
\[
\mathcal{C}_0(f,g)=\{(a,b)\in\mathcal{D}(f,g)\ |\ P(-af-bg)=0\}
\]
where 
\[
\mathcal{D}(f,g)=\big\{(a,b)\in\mathbb{R}^{2}\ |\ ac(f)+bc(g)>0\ \  \mathrm{and}\ \ a+b>0\big\}.
\]

Notice that if $f:\Sigma^+\to\mathbb R$ and $g:\Sigma^+\to\mathbb R$ are both eventually positive and locally H\"older continuous, $f$ has a strong entropy
gap at infinity  and there exists $C$ so that $|f(x)-g(x)|\le C$ for all $x\in\Sigma^+$,
Lemma \ref{eventually to strictly} implies  that $f$ and $g$ are cohomologous to
$\hat f:\Sigma^+\to \mathbb R$ and $\hat g:\Sigma^+\to\mathbb R$ (respectively) which are 
both strictly positive  and locally H\"older continuous, $\hat f$ has a strong entropy
gap at infinity  and there exists $\hat C$ so that $|\hat f(x)-\hat g(x)|\le \hat C$ for all $x\in\Sigma^+$.
Since $\mathcal C(f,g)=\mathcal C(\hat f,\hat g)$, Theorem \ref{Manhattan curve} follows from
the following stronger statement for strictly positive functions.

\medskip\noindent
{\bf Theorem \ref{Manhattan curve}*:} {\em
Suppose that $(\Sigma^{+},\sigma)$ is a topologically mixing, one-sided
countable Markov shift with (BIP), $f:\Sigma^{+}\to\mathbb{R}$ is
locally H\"older continuous, strictly positive and has a strong entropy
gap at infinity and $g:\Sigma^{+}\to\mathbb{R}$ is also strictly
positive and locally H\"older continuous. If there exists $C>0$ so
that $|f(x)-g(x)|<C$ for all $x\in\Sigma^{+}$, then 
\begin{enumerate}
\item $(\delta(f),0),\ (0,\delta(g))\in\mathcal{C}_0(f,g)$. 
\item If $(a,b)\in\mathcal{D}(f,g)$, there exists a unique $t>\frac{d(f)}{a+b}$
so that $(ta,tb)\in\mathcal{C}_0(f,g)$. 
\item $\mathcal{C}_0(f,g)$ is an analytic curve. 
\item $\mathcal{C}_0(f,g)$ is strictly convex, unless 
\begin{equation}\label{eq:sameperiods}
S_{n}f(x)=\frac{\delta(g)}{\delta(f)}S_{n}g(x)
\end{equation}
for all $x\in\mathrm{Fix}^{n}$ and $n\in\mathbb{N}$. 
\end{enumerate}
Moreover, the tangent line to $\mathcal{C}_0(f,g)$ at $(a,b)$ has
slope 
\[
s(a,b)=-\frac{\int_{\Sigma^{+}}g\ d\mu_{-af-bg}}{\int_{\Sigma^{+}}f\ d\mu_{-af-bg}}.
\]
}

\begin{proof}
By definition, $(\delta(f),0)$ and $(0,\delta(g))$ lie on $\mathcal{C}_0(f,g)$
so (1) holds.

Notice that, since $\big| S(f,a)-S(g,a)\big|\le C$ for all $a\in\mathcal A$,
$d(f)=d(g)$ and $g$ also has a strong entropy gap at infinity.
Moreover, if $(a,b)\in\mathcal D(f,g)$, then $af+bg$ is strictly positive, has
a strong entropy gap at infinity and 
\[
d(af+bg)=\frac{d(f)}{a+b}.
\]
Lemma \ref{gap and pressure} then  implies that
if $(a,b)\in \mathcal D(f,g)$, then $t\to P(-t(af+bg))$ is proper and strictly
decreasing on $(\frac{d(f)}{a+b},\infty)$, so there exists
a unique $t>\frac{d(f)}{a+b}$ so that $P(-t(af+bg))=0$.
Thus, (2) holds.

Lemma \ref{equilibrium state exists} implies that there is an equilibrium
state $\mu_{-af-bg}$ for $-af-bg$ and that \hbox{$\int_{\Sigma^+} (-af-bg)\ d\mu_{-af-bg}$} is finite. Notice that if $(c,d)\in\mathcal{D}(f,g)$, then the ratio
$\frac{cf+dg}{af+bg}$ is bounded, this implies that $\int_{\Sigma^{+}}(cf+dg)\ d\mu_{-af-bg}$
is also finite. Theorem \ref{pressure analytic} then implies that
if $(a,b)\in\mathcal{D}(f,g)$, then 
\[
\frac{\partial}{\partial a}P(-af-bg)=\int_{\Sigma^{+}}-f\ d\mu_{-af-bg}
\]
and 
\[
\frac{\partial}{\partial b}P(-af-bg)=\int_{\Sigma^{+}}-g\ d\mu_{-af-bg}.
\]
Since $f$ is strictly positive, \hbox{$\int_{\Sigma^{+}}-f\ d\mu_{-af-bg}$}
is non-zero, so
$P$ is a submersion on $\mathcal{D}(f,g)$. The
implicit function theorem then implies that 
\[
\mathcal{C}_0(f,g)=\{(a,b)\in\mathcal{D}(f,g)\ |\ P(-af-bg)=0\}
\]
is an analytic curve and that if $(a,b)\in\mathcal{C}_0(f,g)$ then
the slope of the tangent line to $\mathcal{C}_0(f,g)$ at $(a,b)$ is
given by 
\[
s(a,b)=-\frac{\int_{\Sigma^{+}}g\ d\mu_{-af-bg}}{\int_{\Sigma^{+}}f\ d\mu_{-af-bg}}.
\]

Since $P$ is convex, see Sarig \cite[Proposition 4.4]{sarig-2009},
$\mathcal{C}_0(f,g)$ is convex. A convex analytic curve is strictly
convex if and only if it is not a line. So it remains to show that
$f$ and $g$ satisfy equation (\ref{eq:sameperiods}) if and only if $\mathcal{C}_0(f,g)$ is a straight
line.

If $\mathcal{C}_0(f,g)$ is a straight line, then by (1) it has
slope $-\frac{\delta(f)}{\delta(g)}$. In particular, 
\begin{equation}
	-s(\delta(f),0)=\frac{\delta(f)}{\delta(g)}=\frac{\int_{\Sigma^{+}}g\ d\mu_{-\delta(f)f}}{\int_{\Sigma^{+}}f\ d\mu_{-\delta(f)f}}=\frac{\int_{\Sigma^{+}}g\ d\mu_{-\delta(g)g}}{\int_{\Sigma^{+}}f\ d\mu_{-\delta(g)g}}.\label{slope value}
\end{equation}

By definition, 
\[
h_{\sigma}(\mu_{-\delta(g)g})-\delta(g)\int_{\Sigma^{+}}g\ d\mu_{-\delta(g)g}=0
\]
so, applying equation (\ref{slope value}), we see that 
\[
h_{\sigma}(\mu_{-\delta(g)g})-\delta(f)\int_{\Sigma^{+}}f\ d\mu_{-\delta(g)g}=\delta(g)\int_{\Sigma^{+}}g\ d\mu_{-\delta(g)g}-\delta(f)\int_{\Sigma^{+}}f\ d\mu_{-\delta(g)g}=0
\]
Since $P(-\delta(f)f)=0$, this implies that $\mu_{-\delta(g)g}$
is an equilibrium state for $-\delta(f)f$. Therefore, by uniqueness
of equilibrium states we see that $\mu_{-\delta(f)f}=\mu_{-\delta(g)g}$.
Sarig \cite[Thm. 4.8]{sarig-2009} showed that this only happens when
$-\delta(f)f$ and $-\delta(g)g$ are cohomologous, so the Livsic
Theorem (Theorem \ref{livsic})
implies that this occurs if and only if 
\[
S_{n}f(x)=\frac{\delta(g)}{\delta(f)}S_{n}g(x)
\]
for all $x\in\mathrm{Fix}^{n}$ and $n\in\mathbb{N}$. We have completed
the proof.
\end{proof}

\section{Background for applications}

In this section, we recall the background material that we will need to construct the roof functions
described in Theorem \ref{Roof Properties}.  We will also recall the more general definition of cusped  Anosov
representations of geometrically finite Fuchsian groups into $\mathsf{SL}(d,\mathbb R)$. In the next
section, we will see that Theorem D also extends to this setting.

\subsection{Linear algebra}
It will be useful to first recall some standard Lie-theoretic  notation. Let 
$$\mathfrak{a}=\{ (a_1,\ldots,a_d)\in\mathbb R^d\  |\ a_1+\ldots+a_d=0\}$$  
be the standard Cartan algebra for $\mathsf{SL}(d,\mathbb R)$ and let
$$\mathfrak{a}^+=\{ (a_1,\ldots,a_d)\in\mathfrak{a}\  |\ a_1\ge\cdots\ge a_d\}$$ 
be the standard choice of positive Weyl chamber.
Let $\mathfrak{a}^*$ be the space of linear functionals on $\mathfrak{a}$. For all $k\in\{1,\ldots, d-1\}$,
let $\alpha_k:\mathfrak a\to\mathbb R$
be given by $\alpha_k(\vec a)=a_k-a_{k+1}$. Then $\{\alpha_1,\ldots,\alpha_{d-1}\}$ span $\mathfrak{a}^*$
and are the simple roots determining the Weyl chamber $\mathfrak{a}^+$.
It is also natural to consider the fundamental weights $\omega_k\in\mathfrak{a}^*$ given by $\omega_k(\vec a)=a_1+\cdots+a_k$.
Notice that $\{\omega_1,\ldots,\omega_{d-1}\}$ is also a basis for $\mathfrak{a}^*$.

If $A\in\mathsf{SL}(d,\mathbb R)$, let
$$\lambda_1(A)\ge \lambda_2(A)\ge\cdots\ge\lambda_d(A)$$
denote the moduli of the generalized eigenvalues of $A$ and let
$$\sigma_1(A)\ge \sigma_2(A)\ge\cdots\ge\sigma_d(A)$$
 be the singular values of $A$. The {\em Jordan projection}
 $$\ell:\mathsf{SL}(d,\mathbb R)\to\mathfrak a^+\ \mathrm{is\ given\ by}\ \ell(A)=(\log \lambda_1(A), \ldots,\log \lambda_d(A))$$
and the {\em Cartan projection} 
$$\kappa:\mathsf{SL}(d,\mathbb R)\to\mathfrak a^+\ \mathrm{is\ given\ by}\ \kappa(A)=(\log \sigma_1(A),\ldots,\log\sigma_d(A)).$$

If $\alpha_k(\ell(A))>0$, then there is a well-defined {\em attracting $k$-plane} which is the plane
spanned by the generalized eigenspaces with eigenvalues of modulus at least $\lambda_k(A)$.
Recall that the Cartan decomposition
of  \hbox{$A\in\mathsf{SL}(d,\mathbb R)$} has the form $A=KDL$ where $K,L\in\mathsf{SO}(d)$
and $D$  is the diagonal matrix with diagonal entries $d_{ii}=\sigma_i(A)$.
If $\alpha_k(A)>0$, then the $k$-flag $U_k(A)=K\big(\langle e_1,\ldots, e_k\rangle\big)$ is well-defined, and
is the $k$-plane spanned by the $k$ longest axes of the ellipsoid $A(S^{d-1})$.
(Notice that $U_k(A)$ is not typically the attracting $k$-plane even when $\alpha_k(\ell(A))>0$.)

\subsection{Cusped Anosov representations of geometrically finite Fuchsian groups}
Suppose that $\Gamma\subset\mathsf{PSL}(2,\mathbb R)$ is a torsion-free geometrically finite Fuchsian group, which is not convex cocompact, and let $\Lambda(\Gamma)$
be its limit set in $\partial\mathbb H^2$.

We say that a representation $\rho:\Gamma\to\mathsf{SL}(d,\mathbb R)$ is {\em cusped $P_k$-Anosov}, for some $1\le k\le d-1$, if there exist
continuous $\rho$-equivariant maps $\xi_\rho^k:\Lambda(\Gamma)\to \mathrm{Gr}_k(\mathbb R^d)$ and
 $\xi^{d-k}_\rho:\Lambda(\Gamma)\to \mathrm{Gr}_{d-k}(\mathbb R^d)$ so that
 \begin{enumerate}
 \item
 $\xi_\rho^k$ and $\xi_\rho^{d-k}$ are {\em transverse}, i.e. if $x\ne y\in\Lambda(\Gamma)$, then
 $$\xi_\rho^k(x)\oplus\xi_\rho^{d-k}(y)=\mathbb R^d.$$
\item
 $\xi_\rho^k$ and $\xi_\rho^{d-k}$ are {\em strongly dynamics preserving}, i.e. if
 $j$ is $k$ or $d-k$ and $\{\gamma_n\}$ is a sequence in $\Gamma$ so
 that $\gamma_n(0)\to x\in\Lambda(\Gamma)$ and $\gamma_n^{-1}(0)\to y\in\Lambda(\Gamma)$, then
 if $V\in\mathrm{Gr}_j(\mathbb R^d)$ and $V$ is transverse to $\xi_\rho^{d-j}(y)$, then $\rho(\gamma_n)(V)\to\xi_\rho^j(x)$.
 \end{enumerate}
 
 The original definition of a cusped $P_k$-Anosov representation in \cite{CZZ} is given in terms of a flow space, as in
 Labourie's original definition \cite{labourie-invent}. The characterization
 we give here is a natural generalization of characterizations of Gu\'eritaud-Guichard-Kassel-Wienhard \cite{GGKW},
 Kapovich-Leeb-Porti \cite{KLP} and Tsouvalas \cite{kostas} in  the traditional setting. Our cusped $P_k$-Anosov representations are examples of
 the relatively Anosov representations considered by Kapovich-Leeb \cite{KL} and the relatively dominated
 representations considered by Zhu \cite{feng}.
  
The following crucial properties of cusped $P_k$-Anosov representations are established in Canary-Zhang-Zimmer \cite{CZZ}.
(Several of these properties also follow from work of Kapovich-Leeb \cite{KL} and Zhu \cite{feng} once one
establishes that our representations fit into their framework.)
If $\rho:\Gamma\to\mathsf{SL}(d,\mathbb R)$ is cusped $P_k$-Anosov, we define the space of
type-preserving deformations
$$\mathrm{Hom}_{tp}(\rho)\subset\mathrm{Hom}(\Gamma,\mathsf{SL}(d,\mathbb R))$$
to be the space of representations  $\sigma$ such that  if $\alpha\in\Gamma$ is parabolic,
then $\sigma(\alpha)$ is conjugate to $\rho(\alpha)$.

\begin{thm} {\rm (Canary-Zhang-Zimmer \cite{CZZ})}
\label{PkAnosov properties}
If $\Gamma$ is a geometrically finite Fuchsian group and \hbox{$\rho:\Gamma\to\mathsf{SL}(d,\mathbb R)$} is a cusped $P_k$-Anosov
representation, then
\begin{enumerate}
\item
There exist $A,a>0$ so that if $\gamma\in\Gamma$, then
$$Ae^{ad(b_0,\gamma(b_0))}\ge e^{\alpha_k(\kappa(\rho(\gamma)))}\ge \frac{1}{A}e^{\frac{d(b_0,\gamma(b_0))}{a}}$$
where $b_0$ is a basepoint for $\mathbb H^2$.
\item
There exist $B,b>0$ so that if $\gamma\in\Gamma$, then
$$Be^{b t(\gamma)}\ge e^{\alpha_k(\ell(\rho(\gamma)))}\ge \frac{1}{B}e^{\frac{t(\gamma)}{b}}$$
where $t(\gamma)$ is the translation length of $\gamma$ on $\mathbb H^2$.
\item
The limit maps $\xi_\rho^k$ and $\xi_\rho^{d-k}$ are H\"older continuous.
\item
There exists an open neighborhood $U$ of $\rho$ in $\mathrm{Hom}_{tp}(\rho)$, so that if $\sigma\in U$,
then $\sigma$ is cusped \hbox{$P_k$-Anosov.}
\item
If $\upsilon\in\Gamma$ is parabolic and $j\in\{1,\ldots,d-1\}$, then there exists $c_j(\rho,\upsilon)\in\mathbb Z$ and $C_j(\rho,\upsilon)>0$ so that
$$\big|\alpha_j(\kappa(\rho(\upsilon^n)))-c_j(\rho,\upsilon)\log n\big|<C_j(\rho,\upsilon)$$
for all $n\in\mathbb N$.  Moreover, if $\eta\in\mathrm{Hom}_{tp}(\rho)$, then $c_j(\rho,\upsilon)=c_j(\eta,\upsilon)$.
\item $\rho$ has the \emph{$P_k$-Cartan property}, i.e.  whenever $\{\gamma_n\}$ is a sequence of distinct elements of $\Gamma$
  such that $\gamma_n(b_0)$ converges to $z\in\Lambda(\Gamma)$, 
  then $\xi_\rho^k(z)=\lim U_k(\rho(\gamma_n))$.
 \item
 $\rho$ is $P_{d-k}$-Anosov.
\end{enumerate}
\end{thm}

\subsection{Cusped Hitchin representations}
Canary, Zhang and Zimmer \cite{CZZ} also prove that  cusped Hitchin representations are cusped $P_k$-Anosov for all $k$, i.e they are  cusped Borel Anosov,
in analogy with work of Labourie \cite{labourie-invent} in the uncusped case.
We say that $A\in\mathsf{SL}(d,\mathbb R)$ is {\em unipotent and totally positive} with respect to a
basis $b=(b_1,\ldots,b_d)$ for $\mathbb R^d$, if its matrix representative with respect to this
basis is unipotent, upper triangular, and all the minors which could be positive are positive.
Let $U_{>0}(b)$ denote the set of all such maps. One crucial property here is that
$U_{>0}(b)$ is a semi-group (see Lusztig \cite{lusztig}).

We say that a basis $b=(b_1,\ldots,b_d)$ is {\em consistent} with a pair $(F,G)$ of transverse flags if
$\langle b_i\rangle=F^i\cap G^{d-i+1}$ for all $i$.
A $k$-tuple $(F_1,\ldots,F_k)$ in $\mathcal F_d$ is  {\em positive} if there exists a basis $b$ consistent
with $(F_1,F_k)$ and there exists $\{u_2,\ldots,u_k\}\in U(b)_{>0}$ so that 
$F_i=u_i\cdots u_2F_1$ for all $i=2,\ldots, d$.

If $X$ is a subset of $S^1$, we say that  a map $\xi:X\to\mathcal F_d$ is {\em positive} if whenever
$(x_1,\ldots,x_k)$ is a consistently ordered $k$-tuple in $X$ (ordered either clockwise or counter-clockwise),
then $(\xi(x_1),\ldots,\xi(x_k))$ is a positive $k$-tuple of flags.

A {\em cusped Hitchin representation} is a representation $\rho:\Gamma\to \mathsf{SL}(d,\mathbb R)$
such that if $\gamma\in\Gamma$ is parabolic, then $\rho(\gamma)$ is a unipotent element with a single Jordan block
and there exists a $\rho$-equivariant positive map $\xi_\rho:\Lambda(\Gamma)\to\mathcal F_d$.
(In fact, it suffices to define $\xi_\rho$ on the subset 
$\Lambda_{per}(\Gamma)$ consisting of  fixed points of peripheral elements of $\Gamma$.)

\begin{thm} {\rm (Canary-Zhang-Zimmer \cite{CZZ})}
\label{cusped Hitchin properties}
If $\Gamma$ is a geometrically finite Fuchsian group and
\hbox{$\rho:\Gamma\to\mathsf{SL}(d,\mathbb R)$} is a cusped Hitchin representation, then
\begin{enumerate}
\item
$\rho$ is $P_k$-Anosov for all $1\le k\le d-1$.
\item
$\rho$ is irreducible.
\item
If $\alpha\in\Gamma$ is parabolic and $1\le k\le d-1$, then
$c_k(\rho,\alpha)=2$.
\end{enumerate}
\end{thm}

We remark that Sambarino \cite{sambarino-closures} has independently established that $\rho$ is irreducible and that
Kapovich-Leeb indicate in \cite{KL} that they can prove $\rho$ is Borel Anosov.

\subsection{Codings for geometrically finite Fuchsian groups}
A  torsion-free convex cocompact Fuchsian group admits a finite Markov shift which codes the recurrent portion
of its geodesic flow. The most basic such coding is the Bowen-Series coding \cite{bowen-series}. However,
if the group is geometrically finite, but not convex cocompact, this coding is not well-behaved. In this case one must instead consider the countable Markov shifts constructed by
Dal'bo-Peign\'e \cite{dalbo-peigne}, if the quotient has infinite area, and Stadlbauer \cite{stadlbauer} and Ledrappier-Sarig 
\cite{ledrappier-sarig}, if the quotient has finite area.

We summarize the crucial properties of these Markov shifts in the following theorem and will give a brief
description of each coding.

\begin{thm} {\rm (Dal'bo-Peign\'e \cite{dalbo-peigne}, Ledrappier-Sarig \cite{ledrappier-sarig}, Stadlbauer \cite{stadlbauer})}
\label{coding properties}
Suppose that $\Gamma$ is a torsion-free geometrically finite, but not cocompact, Fuchsian group. 
There exists a topologically mixing  Markov shift $(\Sigma^+,\mathcal A)$ with countable alphabet $\mathcal A$ with (BIP) which codes the recurrent portion of the geodesic flow 
on $T^1(\mathbb H^2/\Gamma)$.
There exist maps 
$$G:\mathcal A\to\Gamma,\ \ \omega:\Sigma^+\to\Lambda(\Gamma),\ \ r:\mathcal A\to \mathbb N,\ \ \mathrm{and} \ \ s:\mathcal A\to\Gamma$$
with the following properties.
\begin{enumerate}
\item
$\omega$ is locally H\"older continuous and finite-to-one, and $\omega(\Sigma^+)=\Lambda_c(\Gamma)$, i.e.  the complement in $\Lambda(\Gamma)$ of the set of fixed points
of parabolic elements of $\Gamma$. 
Moreover, $\omega(x)=G(x_1)\omega(\sigma(x))$ for every
$x\in\Sigma^+$.
\item
If $x\in \mathrm{Fix}^n$, then $\omega(x)$ is the attracting fixed point of $G(x_1)\cdots G(x_n)$. Moreover, if $\gamma\in\Gamma$ is hyperbolic,
then there exists  $x\in\mathrm{Fix}^n$ (for some $n$) so that $\gamma$ is conjugate to $G(x_1)\cdots G(x_n)$ and $x$ is unique up to shift.
\item
There exists $Q\in\mathbb N$ such that $1\le \#(r^{-1}(n))\le Q$ for all $n\in\mathbb N$.
\item
There exists a finite collection $\mathcal P$ of 
parabolic elements of $\Gamma$, a finite collection $\mathcal R$ of
elements of $\Gamma$ such that if $a\in\mathcal A$, then
$s(a)\in\mathcal P\cup\{id\}$ and $G(a)=s(a)^{r(a)-2}g_a$ where $g_a\in\mathcal R$.
\item
Given a basepoint $b_0\in\mathbb H^2$, there exists $L>0$ so that if $x\in\Sigma^+$ and $n\in\mathbb N$, then 
$$d\big(G(x_1)\cdots G(x_n)(b_0),\overrightarrow{b_0\omega(x)}\big)\le L.$$
\end{enumerate}
\end{thm}

If $\Gamma$ is convex cocompact, then one may use the Bowen-Series \cite{bowen-series} coding $(\Sigma^+,\sigma)$
which we briefly recall to set the scene for the more complicated codings we will need in the non-convex cocompact setting.
One begins  with a fundamental domain $D_0$ for $\Gamma$, containing the basepoint $b_0$,
all of whose vertices lie in $\partial\mathbb H^2$,
so that the set of face pairings  $\mathcal A$ of $D_0$ is a minimal symmetric generating set for $\Gamma$. 
The classical Bowen-Series coding on the alphabet $\mathcal
A$ can be constructed from a ``cutting
sequence'' which records the intersections  $(t_k)$ of a geodesic ray  $\overleftrightarrow{b_0z}$ which
intersects $D_0$, where $z\in\Lambda(\Gamma)$, with edges of
translates of $D_0$ so that the geodesic is entering $\gamma_k(D_0)$  as it passes through $t_k$.
The classical Bowen-Series coding for $\overleftrightarrow{b_0z}$ is given by $(x_k)=(\gamma_k\gamma_{k-1}^{-1})$.
Each $\gamma_k\gamma_{k+1}^{-1}$ is a face-pairing, hence this alphabet
$\mathcal A$ is a finite generating set for $\Gamma$. 
Thus one obtains a map $G:\mathcal A\to\Gamma$, the map $\omega$ simply takes the word encoding the geodesic
ray $\overrightarrow{b_0z}$ to $z$. Moreover, $r(a)=1$ and $s(a)=id$ for all $a\in\mathcal A$.
A word $x$ in $\mathcal A$ is allowable in this coding if and only if $G(x_{i+1})\neq G(x_{i})^{-1}$ for any $i$.

If $\Gamma$ is geometrically finite and has infinite area quotient, then we may use the Dal'bo-Peign\'e coding \cite{dalbo-peigne}.
Roughly, the Dal'bo-Peign\'e coding coalesces all powers of a parabolic  generator  in the Bowen-Series coding. This alteration allows
$\omega$ to be locally H\"older continuous.
Here we may begin with fundamental domain $D_0$ for $\Gamma$, containing the origin $0$ in the Poincar\'e disk model,
all of whose vertices lie in $\partial\mathbb H^2$,
so that the set of face pairings  $\mathcal A_0$ of $D_0$ is a minimal symmetric generating set for $\Gamma$ and such
that every parabolic element of $\Gamma$ is conjugate to an element of $\mathcal A_0$. Let $\mathcal P$ denote the
parabolic elements of $\mathcal  A_0$. We let 
$$\mathcal A=\mathcal A_0\cup\{p^n\\ |\ n\ge 2,\ \ p\in\mathcal P\}.$$
In all cases, $G(a)=a$.
If $a=p^n$ for some $p\in\mathcal P$, then $r(a)=n+1$, $s(a)=p$ and $g_a=p$, while if not we set
$r(a)=1$, $s(a)=id$ and $g_a=a$. A word $x$ in $\mathcal A$ is allowable in this coding if and only if for any $i$, $G(x_{i+1})\neq G(x_{i})^{-1}$
and if $s(x_i)\in\mathcal P$, then $s(x_{i+1})\notin\{s(x_i),s(x_i)^{-1}\}$. 
For a discussion of this coding in our language, see Kao \cite{kao-manhattan}.

If $\Gamma$ is geometrically finite and has a finite area quotient then one cannot use the Dal'bo-Peign\'e coding, 
since there is not a minimal symmetric generating set which contains elements conjugate to every primitive
parabolic element of $\Gamma$. Stadlbauer \cite{stadlbauer} and Ledrappier-Sarig \cite{ledrappier-sarig} construct
a (more complicated) coding in this setting which has the same flavor and coarse behavior as the Dal'bo-Peign\'e coding.
One begins with a Bowen-Series coding of $\Gamma$ with alphabet $\mathcal A_0$.  Let $\mathcal C$ denote a set
of minimal length conjugates of primitive parabolic elements.
They then choose a sufficiently large even number $2N$ so that the length of every element of $\mathcal C$ 
divides $2N$ and let $\mathcal  P$ be the collection of powers of elements of $\mathcal C$ of length 
exactly $2N$. 
Let $\mathcal A_1$ be the set of all strings $(b_0,b_1,\ldots,b_{2N})$ in $\mathcal A_0$ so that 
$b_0b_1\cdots b_{2N}$ is freely reduced in $\mathcal A_0$ and so that
neither $b_1b_2\cdots b_{2N}$ or $b_0b_1\cdots b_{2N-1}$ lies in $\mathcal P$.
Let $\mathcal A_2$ be the set of all freely reduced strings  of the form $(b,\upsilon^t,\upsilon_1,\cdots, \upsilon_{k-1}, c)$
where $b\in\mathcal A_0-\{\upsilon_{2N}\}$, 
$\upsilon=\upsilon_1\cdots \upsilon_{2N}\in\mathcal P$, 
$\upsilon_i\in\mathcal  A_0$ for all $i$, 
$t\in\mathbb N$ and $c\in\mathcal A_0-\{\upsilon_{k}\}$.  Let $\mathcal A=\mathcal A_1\cup\mathcal A_2$.
If  \hbox{$a=(b_0,b_1,\ldots,b_{2N})\in\mathcal A_1$}, then $G(a)=b_1$, $r(a)=1$, $s(a)=id$ and $g_a=b_1$,
while if  \hbox{$a=(b,\upsilon^t,\upsilon_1\cdots \upsilon_{k-1}, c)$}, 
then let $G(a)=\upsilon^{t-1}\upsilon_1\cdots \upsilon_{k-1}$, 
\color{black}
$r(a)=t+1$, $s(a)=\upsilon$ and $g_a=\upsilon_{1}\cdots \upsilon_{k-1}$. The set of allowable words is defined so that if $x\in \mathrm{Fix}^n$, then
$G(x_1)\cdots G(x_n)$ cannot be a parabolic element of $\Gamma$.
(For a more detailed description see
Stadlbauer \cite{stadlbauer}, Ledrappier-Sarig \cite{ledrappier-sarig} or Bray-Canary-Kao \cite{BCK}.)

\subsection{Busemann and Iwasawa cocycles}
\label{iwasawa}

We will use the Busemann cocycle to define our roof functions. We first develop the theory we will need in
the simpler case where $\rho$ is cusped $P_k$-Anosov for all $k$. This theory will suffice for all our application to
cusped Hitchin representations, so one may ignore the discussion of partial flag varieties and partial Iwasawa cocycles
on a first reading.

Quint \cite{quint-ps} introduced a vector valued smooth cocycle, called the {\em Iwasawa cocycle},
$$B:\mathsf{SL}(d,\mathbb R)\times\mathcal F_d\to \mathfrak{a}$$
where $\mathcal F_d$ is the space
of (complete) flags in $\mathbb R^d$. Let $F_0$ denote the standard flag 
$$F_0=\left(\langle e_1\rangle, \langle e_1,e_2\rangle,\ldots, \langle e_1,\ldots,e_{d-1}\rangle\right).$$
We can write any $F\in\mathcal F_d$ as $F=K(F_0)$ where $K\in\mathsf{SO}(d)$. 
If $A\in\mathsf{SL}(d,\mathbb R)$ and $F\in\mathcal F_d$, the Iwasawa decomposition of $AK$ has the form $QZU$ where $Q\in\mathsf{SO}(d)$,
$Z$ is a diagonal matrix with non-negative entries, and $U$ is unipotent and upper triangular. Then
$B(A,F)=(\log z_{11},\ldots,\log z_{dd})$.

One may  check that it satisfies the following cocycle property (see Quint \cite[Lemma 6.2]{quint-ps}):
$$B(ST,F)=B(S,TF)+B(T,F).$$

If $A$ is loxodromic (i.e. $\alpha_k(\ell(A))>0$ for all $k$), then the set of attracting $k$-planes forms a flag
$F_A$, called  the attracting flag
of $A$. In this case, 
\begin{equation}
\label{loxfact}
B(A,F_A)=\ell(A)
\end{equation}
since if $F_A=K_A(F_0)$, then $AK_A$ is upper triangular and the diagonal entries are the  eigenvalues with their moduli in descending order.
(See Lemma 7.5 in Sambarino \cite{sambarino-quantitative}.)

The Iwasawa cocycle is also closely related to the singular value decomposition, also known as the Cartan
decomposition. 
If $A$ is Cartan loxodromic (i.e. $\alpha_k(\kappa(A))>0$ for all
$k$), then the flag $U(A)=\{U_k(A)\}$ is well-defined.
If $W$ is the involution taking $e_i$ to $e_{d-i+1}$ and $A$ has Cartan decomposition $A=KDL$, 
then $A^{-1}$ has Cartan decomposition
$$A^{-1}=\big( L^{-1}W\big)\  \big(W D^{-1}W\big)\ \big(WK^{-1}\big).$$
So  if $S(A)=U(A^{-1})$,
one may check that $B(A,S(A))=\kappa(A)$.
Moreover, the Cartan decomposition bounds the Iwasawa cocycle, specifically
$$||B(A,F)||\le ||\kappa(A)||$$
(see  Benoist-Quint \cite[Corollary 8.20]{benoist-quint-book}).

We will make use of the following close relationship between the Iwasawa cocycle and
the Cartan projection. 

\begin{lem} {\rm (Quint \cite[Lemma 6.5]{quint-ps})}
\label{quintlemma}
For any $\epsilon\in (0,1)$, there exists $C>0$ so that if $A\in\mathsf{SL}(d,\mathbb R)$, $F\in\mathcal F_d$, $\sigma_k(A)>\sigma_{k+1}(A)$
and $\angle\Big(F^k,U_{d-k}(A^{-1})\Big)\ge\epsilon$,
then
$$\big|\omega_k(B(A,F))-\omega_k(\kappa(A))\big|\le C.$$
\end{lem}

Given a representation $\rho:\Gamma\to\mathsf{SL}(d,\mathbb R)$  of a geometrically finite Fuchsian group $\Gamma$ and a $\rho$-equivariant
map $\xi_\rho:\Lambda(\Gamma)\to \mathcal F_d$ we define its associated {\em Busemann cocycle}
$$\beta_\rho:\Gamma\times\Lambda(\Gamma)\to\mathfrak{a}$$
by letting 
$$\beta_\rho(\gamma,x)=B\left(\rho(\gamma),\rho(\gamma^{-1})(\xi_\rho(x))\right).$$ 
The Busemann cocycle was first defined by Quint \cite{quint-ps} and was previously used to powerful effect  in the setting of uncusped
Hitchin representations by Sambarino \cite{sambarino-indicator}, Martone-Zhang \cite{martone-zhang} and Potrie-Sambarino \cite{potrie-sambarino}.

\begin{lem}
\label{busemannfacts} 
If $\rho:\Gamma\to\mathsf{SL}(d,\mathbb R)$ is a representation of a geometrically finite Fuchsian group $\Gamma$ and 
$\xi_\rho:\Lambda(\Gamma)\to\mathcal F_d$ is a $\rho$-equivariant map, then
$\beta_\rho$ satisfies the cocycle property
$$\beta_\rho(\alpha\gamma,z)=\beta_\rho(\alpha,z)+\beta_\rho(\gamma,\alpha^{-1}(z))$$
for all $\alpha,\gamma\in\Gamma$ and $z\in \Lambda(\Gamma)$.

Moreover, if  $\rho(\gamma)$ is loxodromic and $\xi_\rho(\gamma^+)$ is the attracting flag of $\rho(\gamma)$, then
$$\beta_\rho(\gamma,\gamma^+)=\ell(\rho(\gamma)).$$
\end{lem}

\begin{proof}
First notice that
\begin{eqnarray*}
\beta_\rho(\alpha\gamma,z) & = & B\left(\rho(\alpha)\rho(\gamma),\rho(\gamma^{-1})\rho(\alpha^{-1})(\xi_\rho(z))\right)\\
& = & B\left(\rho(\alpha),\rho(\alpha)^{-1}(\xi_\rho(z))\right)+ B\left(\rho(\gamma),\rho(\gamma^{-1})\rho(\alpha^{-1})(\xi_\rho(z))\right)\\
& = & \beta_\rho(\alpha,z)+\beta_\rho(\gamma,\alpha^{-1}(z)).
\end{eqnarray*}
Then observe that
$$\beta_\rho(\gamma,\gamma^+)=B\left(\rho(\gamma),\rho(\gamma^{-1})(\xi_\rho(\gamma^+))\right)=B(\rho(\gamma),\xi_\rho(\gamma^+)).$$
Since we have assumed that $\xi_\rho(\gamma^+)$ is the attracting flag of $\rho(\gamma)$, we
may apply Equation (\ref{loxfact}).
\end{proof}

We now generalize the theory developed above to the setting of partial flag varieties. 
If 
\newline
\hbox{$\theta=\{i_1< \cdots <i_r\}\subset \{1,\ldots, d\}$}, then a {\em $\theta$-flag }is a nested collection of vector subspaces of dimension $i_j$ of
the form
$$F=\{0\subset F^{i_1}\subset\cdots\subset F^{i_r}\subset\mathbb R^d\}.$$
The {\em $\theta$-flag variety} $\mathcal F_\theta$ is the set of all $\theta$-flags. Let
$$\mathfrak{a}_\theta=\big\{\vec a\in\mathfrak{a}\ |\ \alpha_k(\vec a)=0\ \mathrm{if}\ k\notin\theta\big\}.$$
There is a unique projection
$$p_\theta:\mathfrak{a}\to\mathfrak{a}_\theta$$
invariant by $\{w\in W\colon w(\frak a_\theta)=\frak a_\theta\}$ where $W$ is the Weyl group acting on $\frak a$ by coordinate permutations.
Benoist and Quint \cite[Section 8.6]{benoist-quint-book} describe a partial Iwasawa cocycle
$$B_\theta:\mathsf{SL}(d,\mathbb R)\times\mathcal F_\theta\to \mathfrak{a}_\theta$$
such that $p_\theta\circ B$ factors through $B_\theta$.

We say that $A\in\mathsf{SL}(d,\mathbb R)$ is $\theta$-proximal if $\alpha_k(\ell(A))>0$ for all $k\in\theta$.
In this case, $A$ has a well-defined attracting $\theta$-flag $F^\theta_A$, and
$$B_\theta(A,F^\theta_A)=p_\theta(\ell(A))$$
In particular,
\begin{equation}
\label{thetafact}
\omega_k(B_\theta(A,F_A^\theta))=\omega_k(\ell(A))
\end{equation}
for all $k\in\theta$.

Given a representation $\rho:\Gamma\to\mathsf{SL}(d,\mathbb R)$  of a geometrically finite Fuchsian group $\Gamma$ and a \hbox{$\rho$-equivariant}
map $\xi_\rho:\Lambda(\Gamma)\to \mathcal F_\theta$ we define its associated {\em $\theta$-Busemann cocycle}
$$\beta_\rho^\theta :\Gamma\times\Lambda(\Gamma)\to\mathfrak{a}_\theta$$
by letting
$$\beta^\theta_\rho(\gamma,z)=B_\theta\left(\rho(\gamma),\rho(\gamma^{-1})(\xi_\rho(z))\right).$$ 
Since $p_\theta$ is linear, Lemma \ref{busemannfacts}
immediately generalizes to give

\begin{lem}
\label{busemannfacts general}
If $\rho:\Gamma\to\mathsf{SL}(d,\mathbb R)$ is a representation of a geometrically finite Fuchsian group $\Gamma$ and 
$\xi:\Lambda(\Gamma)\to\mathcal F_\theta$ is a $\rho$-equivariant map, then
$\beta_\rho^\theta$ satisfies the cocycle property
$$\beta_\rho^\theta(\alpha\gamma,z)=\beta_\rho^\theta(\alpha,z)+\beta_\rho^\theta(\gamma,\alpha^{-1}(z))$$
  for all $\alpha,\gamma\in \Gamma$ and $z\in\Lambda(\Gamma)$.

Moreover, if  $\rho(\gamma)$ is $\theta$-proximal  and $\xi_\rho(\gamma^+)$ is the attracting $\theta$-flag of $\rho(\gamma)$, then
$$\beta^\theta_\rho(\gamma,\gamma^+)=p_\theta(\ell(\rho(\gamma))).$$
In particular,
$$\omega_k(\beta^\theta_\rho(\gamma,\gamma^+))=\omega_k(\ell(\rho(\gamma)))$$
if $k\in\theta$.
\end{lem}

\section{Roof functions for Anosov representations}

If $\theta\subset \{1,\ldots, d-1\}$ is non-empty, we will say that  $\rho:\Gamma\to\mathsf{SL}(d,\mathbb R)$ is cusped $\theta$-Anosov
 if it is cusped $P_k$-Anosov for all $k\in\theta$. We say that $\theta$ is  {\em symmetric} if $k\in\theta$ if and only
if $d-k\in\theta$. It will be natural to always assume that $\theta$ is symmetric, since $\rho$ is cusped $P_k$-Anosov if and only if it is 
cusped $P_{d-k}$-Anosov.
If $\rho:\Gamma\to\mathsf{SL}(d,\mathbb R)$ is a cusped $\theta$-Anosov representation of a geometrically finite
Fuchsian group, we define a vector valued roof function
$$\tau_\rho:\Sigma^+\to\mathfrak{a}_\theta$$ 
by setting 

$$\tau_\rho(x)=\beta^\theta_\rho\big(G(x_1),\omega(x)\big)=B_\theta\Big(\rho(G(x_1)),\rho(G(x_1))^{-1}\big(\xi_\rho(\omega(x))\big)\Big).$$

If $\phi$ is a linear functional on $\mathfrak{a}_\theta$ we define the {\em $\phi$-roof function} $\tau_\rho^\phi=\phi\circ\tau_\rho$.
If $\rho$ is cusped Borel Anosov, i.e. if $\theta=\{1,\ldots,d-1\}$, then $\mathfrak{a}_\theta=\mathfrak{a}$ and $B_\theta=B$ so we are
in the simpler setting described in the first part of Section \ref{iwasawa}.

Recall that the {\em Benoist limit cone} of a representation $\rho:\Gamma\to\mathsf{SL}(d,\mathbb R)$  is given by
$$\mathcal{B}(\rho)=\overline{\bigcap_{n\ge 0}\bigcup_{||\kappa(\rho(\gamma))||\ge n} \mathbb R_+\kappa(\rho(\gamma))}\subset\mathfrak{a}^+.$$
Benoist \cite{benoist-asymptotic} showed that if $\Gamma$ is Zariski dense, then $\mathcal{B}(\rho)$ is convex and has non-empty interior. 
It is natural to consider linear functionals which are positive on the Benoist limit cone
$$\mathcal{B}(\rho)^+=\Big\{\phi\in\mathfrak{a}^*\ | \ \phi\Big(\mathcal{B}(\rho)-\{\vec 0\}\Big)\subset (0,\infty)\Big\}.$$
Note that if $\phi\in\mathcal B(\rho)^+$, then there is a constant
$c$ such that $\phi(v)>c\|v\|$ for all $v\in\mathcal B(\rho)$.

We will in general consider roof functions associated to linear functionals in $\mathfrak{a}_\theta^*\cap\mathcal B(\rho)^+$.
Recall that $\mathfrak{a}_{\theta}^*$ is spanned by $\{\omega_k\ |\ k\in\theta\}$.
So if $\{1,d-1\}\subset\theta$ and $\rho$ is cusped $\theta$-Anosov (i.e. if $\rho$ is  cusped $P_1$-Anosov), 
then $\omega_1$ and 
the Hilbert length functional
$\alpha_H=\omega_1+\omega_{d-1}$ 
both lie
in $\mathfrak{a}_\theta^*\cap\mathcal B(\rho)^+$.
If $\{1,2\}\subset\theta$, then $\alpha_1=\omega_2-2\omega_1\in \mathfrak{a}_\theta^*\cap\mathcal B(\rho)^+$, and, more generally,
if $\{k-1,k,k+1\}\subset\theta$ , then $\alpha_k=-\omega_{k+1}+2\omega_k-\omega_{k-1}\in \mathfrak{a}_\theta^*\cap\mathcal B(\rho)^+$, if $\rho$ is cusped $\theta$-Anosov.
Finally, if $\theta=\{1,\ldots,d-1\}$ (i.e. $\rho$ is cusped Borel
Anosov), then
$$\Delta=\big\{ a_1\alpha_1+\ldots+a_{d-1}\alpha_{d-1}\ |\ \ a_i\ge0\ \
\forall i,\ \sum_{i=1}^{d-1} a_i>0\big\}\subset\mathfrak{a}_\theta^*\cap\mathcal B(\rho)^+=\mathcal B(\rho)^+.$$

\medskip\noindent
{\bf Theorem \ref{Roof Properties}*:} {\em
Suppose that $\Gamma$ is a torsion-free geometrically finite, but not  convex cocompact, Fuchsian group, $\theta\subset \{1,\ldots,d-1\}$ is non-empty and symmetric, and
$\rho:\Gamma\to\mathsf{SL}(d,\mathbb R)$ is cusped $\theta$-Anosov.
If $\phi\in\mathfrak{a}_\theta^*\cap\mathcal B(\rho)^+$, then $\tau_\rho^\phi:\Sigma^+\to\mathbb R$
is a locally H\"older continuous function such that
\begin{enumerate}
\item
If $x=\overline{x_1\cdots x_n}$ is a periodic element of $\Sigma^+$, then
$$S_n\tau_\rho^\phi(x)=\phi\Big(\ell\big(\rho(G(x_1)\cdots G(x_n))\big)\Big).$$
\item
$\tau_\rho^\phi$ is eventually positive.
\item
There exists $C_\rho>0$ such that if $j\in\theta$,
then
$$\Big| \tau_\rho^{\omega_j}(x)-c_j(\rho,s(x_1))\log r(x_1)\Big|\le C_\rho$$
(with the convention that $c_j(\rho,\gamma)=0$ if $\gamma$ is not parabolic).
\item
$\tau_\rho^\phi$ has a strong entropy gap at infinity. Moreover, 
if  $\phi=\sum_{k\in\theta} a_k\omega_k$, then
$$d(\tau_\rho^\phi)=\frac{1}{c(\rho, \phi)}$$
where
$$c(\rho,\phi)=\inf\Big\{\sum_{k\in\theta}a_kc_k(\rho,\upsilon)\ |\ \upsilon\in\Gamma\ \ \mathrm{parabolic}\Big\}.$$
\item 
If $\eta\in\mathrm{Hom}_{tp}(\rho)$ is also $P_k$-Anosov and $\phi\in\mathcal B(\eta)^+$, then there
exists $C>0$ so that
$$|\tau_\rho^\phi(x)-\tau_\eta^\phi(x)|\le C$$
for all $x\in\Sigma^+$. 
\item $\tau_\rho^\phi$ is non-arithmetic.
\end{enumerate}
}

\medskip\noindent
{\em Proof of Theorem \ref{Roof Properties}*.}
It  follows immediately from Lemma \ref{busemannfacts general} and Theorem \ref{coding properties} (1) that if $x\in\Sigma^+$, then
\begin{equation}
\label{roof formula}
S_n\tau_\rho(x)=\sum_{j=0}^{n-1}\tau_\rho(\sigma^j(x))=
\beta^\theta_\rho\big(G(x_1)\cdots G(x_m),\omega(x)\big).
\end{equation}
In particular, if $x=\overline{x_1\cdots x_n}\in\Sigma^+$ is periodic,
then, by Lemma \ref{busemannfacts general} and Theorem \ref{coding properties} (2), 
\begin{equation}
\label{periodic roof formula}
\omega_k\big(S_n\tau_\rho(x)\big)=\omega_k\Big(\ell\big(\rho(G(x_1)\cdots G(x_n))\big)\Big),
\end{equation}
for all $k\in\theta$, since $\xi_\rho(\omega(x))$ is the attracting
$\theta$-flag of $\rho(G(x_1)\cdots G(x_n))$. Thus, (1) holds since
$\{\omega_k\ |\ k\in\theta\}$ is a basis for $\mathfrak a_\theta^\ast$ and the map $\phi\to \tau_\phi$ is linear.

If $\phi\circ\tau_\rho$ is not eventually positive, then there exist sequences $\{x_n\}$ in $\Sigma^+$ and 
$\{ m_n\}$ in $\mathbb N$ so that $m_n\to\infty$ and $\phi\big(S_{m_n}\tau_\rho(x_n)\big)<1$ for all $n$.
Let $\gamma_n=G((x_n)_1)\cdots G((x_n)_{m_n})$ and $z_n=\omega(x_n)$. Then
$$\phi\big(\beta^\theta_\rho(\gamma_n,z_n)\big)<1\qquad \mathrm{for}\ \mathrm{all}\ n\in\mathbb N.$$

We may assume that $\{z_n\}$ converges to $z\in\Lambda(\Gamma)$. Theorem
\ref{coding properties} (5) 
implies that there exists $L$ so that 
$d(\gamma_n(b_0),\overrightarrow{b_0z_n})\le L$ for all $n$.
After passing to another subsequence,  we may assume that
$\{\gamma_n^{-1}(b_0)\}$ converges to some $w\in\Lambda(\Gamma)$. 
We pass to another subsequence, so that $\{\gamma_n^{-1}(z_n)\}$ converges to some $x\in\Lambda(\Gamma)$.
Notice that $x\ne w$, since $\overrightarrow{\gamma_n^{-1}(b_0)\gamma_n^{-1}(z_n)}$
converges to
a bi-infinite geodesic joining $w$ to $x$ which lies within $L$ of the basepoint
$b_0$.

Since $\lim \gamma_n^{-1}(b_0)=w$ and
$\rho$ has the $P_k$-Cartan property for all $k\in\theta$ by Theorem \ref{cusped Hitchin properties}(6),
\[
  \lim U_k(
      \rho(
    \gamma_n^{-1}))=\xi^k_\rho(w).
\]
Since 
$\xi_\rho^{d-k}(x)$ 
and
$\xi_\rho^{d-k}(w)$ 
are transverse, there exist
$N\in\mathbb N$ and $\epsilon>0$ so that if $n>N$, then 
$$\angle\big(\xi_\rho^k(\gamma_n^{-1} z_n),U_{d-k}(\rho(\gamma_n)^{-1})\big)\ge\epsilon.$$
Lemma \ref{quintlemma} and the $\rho$-equivariance of the limit map 
$\xi_\rho$ then imply  that there exists $C$ so that
$$|\omega_k(\beta_\rho^\theta(\gamma_n,\xi_\rho(z_n)))-\omega_k(\kappa(\rho(\gamma_n)))|=
|\omega_k(B_\theta(\rho(\gamma_n),\rho(\gamma_n^{-1})(\xi_\rho(z_n))))-\omega_k(\kappa(\rho(\gamma_n)))|\le C$$
for all $k\in\theta$ and all  $n\ge N$. Since $\phi\in\mathfrak{a}_\theta^*$ this implies that there exists $\hat C>0$ such that
$$|\phi(\beta^\theta_\rho(\gamma_n,\xi_\rho(z_n)))-\phi(\kappa(\rho(\gamma_n)))|\le \hat C$$
for all $n\ge N$.

By Theorem \ref{PkAnosov properties}(1),  $\phi(\kappa(\rho(\gamma_n)))\to\infty$, so we have achieved a
contradiction.
Therefore, $\tau_\rho^\phi$ is eventually positive, so (2) holds.

\medskip

In order to establish (3), we first notice that,
since $||B_\theta(A,F)||\le ||\kappa(A)||$ for all $F\in\mathcal F_\theta$, 
$$|\tau_\rho^{\omega_j}(x)|\le C_{x_1}=j||\kappa(\rho(G(x_1)))||$$
for all $x\in\Sigma^+$ and $j\in\theta$. Since our alphabet is infinite and $C_{x_1}\to\infty$ as $r(x_1)\to\infty$,
there is more work to be done.

If $x\in\Sigma^+$ and 
$r(x_1)\ge 2$, then $G(x_1)=\upsilon^{n}g_a$ for some $\upsilon\in\mathcal P$ and $g_a\in \mathcal R$,
where $n=r(x_1)-2$, then
\begin{eqnarray*}
\tau_\rho(x) & = &
\beta^\theta_\rho\big(\upsilon^ng_a,\omega(x)\big)=B_\theta\big(\rho(\upsilon^ng_a),\rho(\upsilon^n g_a)^{-1}(\xi_\rho(\omega(x)))\big)\\
 &= & B_\theta\big(\rho(\upsilon^n) ,\rho(\upsilon^{-n})(\xi_\rho(\omega(x)))\big)+
B_\theta\big(\rho(g_a),\rho(\upsilon^ng_a)^{-1}(\xi_\rho(\omega(x)))\big).
\end{eqnarray*}
Notice that
$$\Big|\omega_j\Big(B_\theta\big(\rho(g_a),\rho(\upsilon^n g_a)^{-1}(\xi_\rho(\omega(x)))\big)\Big)\Big|\le R=\max\big\{ d\|\kappa(\rho(g_a)) \|\ \big|\  g_a\in \mathcal R\big\}$$
for all $j\in\theta$.

Let  $p$ be the fixed point of $\upsilon$ in $\Lambda(\Gamma)$. Notice that, by construction, there exists $\hat a\in\mathcal A$ so
that $G(\hat a)=\upsilon g_a$. Then $X=\omega([\hat a])$ is a compact subset of $\Lambda(\Gamma)-\{p\}$.
Therefore, if $G(x_1)=\upsilon^ng_a$,  $\omega(x)\in \upsilon^{n-1}(X)$, 
so $\upsilon^{-n}(\omega(x))\in \upsilon^{-1}(X)$.
It follows  that there exists $\epsilon=\epsilon(\upsilon)>0$ so that if $G(x_1)=\upsilon^ng_a$ and $n\in\mathbb N$, then
$$\angle\big(\rho(\upsilon^{-n})(\xi_\rho^j(\omega(x))),\xi_\rho^{d-j}(p)\big)\ge \epsilon$$
for all $j\in\theta$.
Lemma \ref{quintlemma} then implies that there exists $D=D(\upsilon,g_a)>0$ so that
$$\Big|\omega_j\big(B_\theta\big(\rho(\upsilon^n),\rho(\upsilon^{-n})(\xi_\rho(\omega(x)))\big)-\omega_j(\kappa(\rho(\upsilon^n)))\Big|\le D.$$
for all $n\in\mathbb N$ and $j\in\theta$.
Theorem \ref{PkAnosov properties} implies that there exists $C=C(\upsilon,g_a)>0$ so that
$$\big|\omega_j(\kappa(\rho(\upsilon^n)))-c_j(\rho,\upsilon)\log n\big|<C$$
for all $n\in\mathbb N$.
By combining, we see that
$$\Big|\omega_j\Big(B_\theta\big(\rho(\upsilon^n),\rho(\upsilon^{-n})(\xi_\rho(\omega(x)))\big)\Big)- c_j(\rho,\upsilon)\log n\Big|\le C+D$$
and hence that
$$\Big| \tau_\rho^{\omega_j}(x)-c_j(\rho,\upsilon)\log \big(r(x_1)-2\big)\Big|\le C+D+R$$
for all $n\in\mathbb N$ and $j\in\theta$.
Since there are only finitely many $\upsilon$ in $\mathcal P$,  and only finitely many 
elements of $\mathcal A$ so that $r(a)\le 2$ we have completed the proof of (3).

\medskip

We next check that $\tau_\rho^\phi$ is locally H\"older continuous.
Since $\omega:\Sigma^+\to \Lambda(\Gamma)$ is locally H\"older continuous,
there exist $Z>0$ and $\zeta>0$ so that if $x_j=y_j$ for all $j\le n$, then
$$d(\omega(x),\omega(y))\le Ze^{-\zeta n}.$$
Since $\xi_\rho:\Lambda(\Gamma)\to\mathcal F_d$ is H\"older,
there exist $D>0$ and $\iota>0$,
so that if $z,w\in\Lambda(\Gamma)$, then
$$d(\xi_\rho(z),\xi_\rho(w))\le Dd(z,w)^\iota$$
Therefore,
$\xi_\rho\circ\omega$ is locally H\"older continuous, i.e. there exists $C$ and $\beta>0$ so that
$$d(\xi_\rho(\omega(x)),\xi_\rho(\omega(y)))\le Ce^{-\beta n}$$
if $x_j=y_j$ for all $j\le n$.

If $a\in\mathcal A$, let
$$D_a=\sup\left\{||D_FB_\theta(\rho(G(a)),\cdot)||\ \Big|\  F\in\mathcal F_\theta\right\}$$
where $D_FB_\theta(\rho(G(a)),\cdot)$ is the derivative at $F$  of $B_\theta(\rho(G(a)),\cdot):\mathcal F_\theta\to\mathfrak{a}_\theta$. 
It follows that if $x_j=y_j$ for all $j\le n$ and $x_1=y_1=a$, then
$$|\tau^{\phi}_\rho(x)-\tau_\rho^{\phi}(y)|\le ||\phi||D_aCe^{-\beta n}$$

Recall that if $x\in\Sigma^+$ and $G(x_1)=\upsilon^mg_a$, then
$$\tau_\rho(x) =
B_\theta\big(\rho(\upsilon^m),\rho(\upsilon^{-m})(\xi_\rho(\omega(x)))\big)+
B_\theta\big(\rho(g_a),\rho(\upsilon^mg_a)^{-1}(\xi_\rho(\omega(x)))\big)$$
and that $\upsilon^{-m}(\omega(x))$ lies in a compact subset  $\upsilon^{-1}(X)$ of $\Lambda(\Gamma)-\{p\}$ (where $p$ is
the fixed point of $\upsilon$).

There exists $c>0$ so that if $x,y\in \upsilon^{-1}(X)$ and $r\in\mathbb N$, then
$$d(\upsilon^r(x),\upsilon^r(y))\le \frac{c}{r^2} d(x,y).$$
Notice that, by the cocycle property for $B_\theta$,
$$B_\theta\big(\rho(\upsilon^m),F\big)=\sum_{j=1}^m B_\theta(\rho(\upsilon),\upsilon^{j-1}(F)).$$
Thus, if
$$\hat D=\hat D(\upsilon)=\sup\left\{||D_FB_\theta(\rho(\upsilon),\cdot)||\ \Big|\  F\in\mathcal F_\theta\right\}$$
then
$$||B_\theta\big(\rho(\upsilon^m),x\big)-B_\theta\big(\rho(\upsilon^m),y\big)||\le \sum_{s=1}^m \hat D\frac{c}{s^2}d(x,y)$$
if $x,y\in \upsilon^{-1}(X)$.
Notice that there exists $T=T(\upsilon)>0$ so that this
series can be bounded above by $T d(x,y)$.
Therefore, if $x_j=y_j$ for all $j=1,\ldots, n$ and $G(x_1)=\upsilon^sg_a$ where $s\ge 1$, then
$$|(\phi\circ\tau_\rho)(x)-(\phi\circ\tau_\rho)(y)|\le (T+R)C||\phi||e^{-\beta n}$$
where
$$R=\sup\left\{||D_FB_\theta(\rho(g_a),\cdot)||\ \Big|\  F\in\mathcal F_d, \ \ g_a\in\mathcal  R\ \right\}.$$

Since there are only finitely many $\upsilon$ in $\mathcal P$ and only finitely many elements of $\mathcal A$ so that $r(a)\le 2$,
$\tau_\rho^\phi$ is locally H\"older continuous.

\medskip

If $\phi=\sum_{k\in\theta}a_k\omega_k$ and $\upsilon\in\mathcal P$,
let 
$$c(\rho,\phi,\upsilon)=\sum_{k\in\theta} a_k c_k(\rho,\upsilon)\qquad\mathrm{and}\qquad c(\rho,\phi)=\inf\{ c(\rho,\phi,\upsilon)\ |\ \upsilon\in\mathcal P\}.$$ 
Notice that $c(\rho,\phi)$ must be positive, since $\phi\in\mathcal B(\rho)^+$. Property (3) then implies that
$$\big|\tau_\rho^\phi(x) -c(\rho,\phi,s(x_1))\log (r(x_1))\big| \le C_\rho ||\phi||$$
for all $x\in\Sigma^+$.
Therefore,
$$\sum_{n=1}^\infty e^{-sC_\rho||\phi||}\frac{1}{n^{s c(\rho,\phi)}}=\sum_{n=1}^\infty e^{-s\big(c(\rho,\phi)\log n+C_\rho ||\phi||\big)}\le Z_1(\tau_\rho^\phi,s)$$
and 
$$Z_1(\tau_\rho^\phi,s)\le \sum_{n=1}^\infty Qe^{-s\big(c(\rho,\phi)\log n-C_\rho ||\phi||\big)}\le  \sum_{n=1}^\infty Qe^{sC_\rho||\phi||}\frac{1}{n^{s c(\rho,\phi)}}$$
if $s>0$. (Recall that if $n\in\mathbb N$, then $1\le\#\{ a\in\mathcal A\ |\ r(a)=n\}\le Q$.)
Therefore, $Z_1(\tau_\rho^\phi,s)$ converges if and only if $s>\frac{1}{c(\rho,\phi)}$, which establishes (4).

\medskip

If $\eta\in\mathrm{Hom}_{tp}(\rho)$ is cusped $\theta$-Anosov and $\phi\in \mathcal B(\eta^+)$, 
then $c_j(\rho,\upsilon)=c_j(\eta,\upsilon)$ for all $j\in\theta$ and $\upsilon\in\mathcal P$.
Property (5) then follows from applying (3) to both $\tau_\rho$ and $\tau_\eta$ and the fact that both $\tau_\rho^\phi$ and
$\tau_\eta^\phi$ are locally H\"older continuous.

\medskip

We may assume that the  Zariski closure $\mathsf{G}$ of $\rho(\Gamma)$ is reductive.
(If it is not reductive, then Gu\'eritaud-Guichard-Kassel-Wienhard \cite[Section 2.5.4]{GGKW} exhibit a representation
$\rho^{ss}:\Gamma\to \sf{SL}(d, \mathbb R)$ so that the Zariski closure  of $\rho^{ss}(\Gamma)$ is 
reductive and $\ell(\rho(\gamma))=\ell(\rho^{ss}(\gamma))$ for all $\gamma\in \Gamma$.)
A result of Benoist-Quint  \cite[Proposition 9.8]{benoist-quint-book} then
implies that the subgroup $\frak h$ of the Cartan algebra $\mathfrak{a}_{\mathfrak{g}}$ of $\mathsf{G}$ 
generated by $\lambda_{\mathsf{G}}(\rho(\Gamma))$ is dense in $\mathfrak{a}_{\mathfrak{g}}$
(where $\lambda_{\mathsf{G}}:\mathsf{G}\to \mathfrak{a}_{\mathfrak{g}}$ is the Jordan projection of $\mathsf{G}$).
Up to conjugation, we may assume that $\mathfrak{a}_{\mathfrak{g}}$ is a sub-algebra of $\mathfrak{a}$ (since $\mathfrak a_{\mathfrak g}$ is an abelian algebra and thus 
is contained in a translate of $\mathfrak a$, which is a maximal abelian sub-algebra of $\mathfrak{sl}(d,\mathbb R)$). Therefore,
the subgroup of $\mathbb R$ generated by $\{\phi\circ\tau_\rho(x)\ |\ x\in\mathrm{Fix}^n\}$, which is just $\phi(\mathfrak h)$,
is dense in $\mathbb R$. 
Thus, we have established (6).
\qed

\medskip

\section{Applications}\label{sec:applications}

\subsection{Anosov representations of geometrically finite Fuchsian groups}

Given Theorem \ref{Roof Properties}*, we can apply our main results to the roof functions of  Anosov representations. 

The following counting result is a strict generalization of Corollary \ref{cusped counting}.
It follows immediately from Theorems \ref{Roof Properties}* and  \ref{thm:countingN}.

\begin{cor}
Suppose that  $\Gamma$ is a torsion-free, geometrically finite, but not convex cocompact, Fuchsian group, $\theta\subset\{ 1,\ldots,d-1\}$ is non-empty and symmetric, and  $\rho:\Gamma\to\mathsf{SL}(d,\mathbb R)$ is cusped $\theta$-Anosov.
If $\phi\in\mathfrak{a}_\theta^*\cap\mathcal B(\rho)^+$, then there exists a unique $\delta_\phi(\rho)>\frac{1}{c(\rho,\phi)}$ so that $P(-\delta_\phi(\rho) \tau_\rho^\phi)=0$ and
$${\displaystyle \lim_{t\to\infty}M_{\phi}(t)\frac{t\delta_\phi(\rho)}{e^{t\delta_\phi(\rho)}}}=1$$
where 
$$M_{\phi}(t)=\#\Big\{[\gamma]\in[\Gamma]\ \big|\ 0< \phi(\ell(\rho(\gamma)))\le t\Big\}.$$
\end{cor}

Similarly, one may combine Theorems \ref{Manhattan curve} and \ref{Roof Properties}* to obtain a generalization of 
Corollary \ref{CuspedManhattan}.

\begin{cor}\label{cor:manhattan_applications}
Suppose that  $\Gamma$ is a torsion-free, geometrically finite, but not convex cocompact Fuchsian group, $\theta\subset\{ 1,\ldots,d-1\}$ is non-empty and symmetric, and  $\rho:\Gamma\to\mathsf{SL}(d,\mathbb R)$ is
cusped $\theta$-Anosov.
If $\eta\in\mathrm{Hom}_{tp}(\rho)$ is also cusped $\theta$-Anosov, $\phi\in\mathfrak{a}_\theta^*\cap\mathcal B(\rho)^+\cap\mathcal B(\eta)^+$, and
$$\mathcal C^{\phi}(\rho,\eta)=\big\{ (a,b)\in \mathcal D(\rho,\eta) \ | \ P(-a\tau_\rho^{\phi}-b\tau_\eta^{\phi})=0\big\}$$
where 
$$\mathcal D(\rho,\eta)=\big\{(a,b)\in\mathbb R^2\ |\  a+b>c(\rho,\phi)\big\},$$
then 
\begin{enumerate}
\item
$\mathcal C^\phi(\rho,\eta)$ is an analytic curve,
\item
$(\delta_{\phi}(\rho),0)$ and $(0,\delta_{\phi}(\eta))$ lie on $\mathcal C^\phi(\rho,\eta)$,
\item
 $\mathcal C^\phi(\rho,\eta)$ is strictly convex, unless 
$$\ell^{\phi}(\rho(\gamma))=\frac{\delta_\phi(\eta)}{\delta_\phi(\rho)}\ell^{\phi}(\eta(\gamma))$$ 
for all 
$\gamma\in\Gamma$,
\item
and the tangent line to $\mathcal C^{\phi}(\rho,\eta)$ at $(\delta_{\phi}(\rho),0)$ has slope
$$s^\phi(\rho,\eta)=-\frac{\int \tau_\eta^{\phi} dm_{-\delta_{\phi}(\rho)\tau^{\phi}_\rho}}{\int \tau_\rho^{\phi}\ 
dm_{-\delta_{\phi}(\rho)\tau^{\phi}_\rho}}.
$$
\end{enumerate}
\end{cor}
 In the setting of the previous corollary, we may define the {\em pressure intersection} \hbox{$I^{\phi}(\rho,\eta)=-s^\phi(\rho,\eta)$} and the {\em renormalized pressure intersection}
$$J^{\phi}(\rho,\eta)=\frac{\delta^{\phi}(\eta)}{\delta^{\phi}(\rho)}I^{\phi}(\rho,\eta).$$
We obtain the following intersection rigidity result which will be used crucially in the construction of pressure metrics. The proof follows at once from statements (3) and (4) in Corollary \ref{cor:manhattan_applications}.

\begin{cor}
Suppose that  $\Gamma$ is a torsion-free, geometrically finite, but not  convex cocompact, Fuchsian group, $\theta\subset\{ 1,\ldots,d-1\}$ is non-empty and symmetric, 
and  $\rho:\Gamma\to\mathsf{SL}(d,\mathbb R)$ is cusped $\theta$-Anosov.
If $\eta\in\mathrm{Hom}_{tp}(\rho)$ is also cusped $\theta$-Anosov and $\phi\in\mathfrak{a}_\theta^*\cap\mathcal B(\rho)^+\cap\mathcal B(\eta)^+$, then
$$J^{\phi}(\rho,\eta)\ge 1$$
with equality if and only if 
$$\ell^{\phi}(\rho(\gamma))=\frac{\delta_\phi(\eta)}{\delta_\phi(\rho)}\ell^{\phi}(\eta(\gamma))$$
for all $\gamma\in\Gamma$. 
\end{cor}

Finally, we derive our equidistribution result, which generalizes Corollary \ref{geometric intersection}. It
follows immediately from Theorems \ref{thm:equid_roof} and \ref{Roof Properties}*.

\begin{cor}
Suppose that  $\Gamma$ is a torsion-free, geometrically finite, but not convex cocompact, Fuchsian group, $\theta\subset\{ 1,\ldots,d-1\}$ is non-empty and symmetric, and  
$\rho:\Gamma\to\mathsf{SL}(d,\mathbb R)$ is cusped $\theta$-Anosov.
If $\eta\in\mathrm{Hom}_{tp}(\rho)$ is also cusped $\theta$-Anosov  and 
$\phi\in\mathfrak{a}_\theta^*\cap\mathcal B(\rho)^+\cap\mathcal B(\eta)^+$, then
$$I^{\phi}(\rho,\eta)=\lim_{T\to\infty} \frac{1}{\#(R_T^{\phi}(\rho))}\sum_{[\gamma]\in R_T^{\phi}(\rho)}
\frac{\ell^{\phi}(\eta(\gamma))}{\ell^{\phi}(\rho(\gamma))}$$
where $R_T(\rho)=\{ [\gamma]\in\Gamma\ |\ 0<\ell^\phi(\rho(\gamma))\le T\}$.
\end{cor}

\subsection{Traditional Anosov representations} 

Andres Sambarino \cite{sambarino-quantitative,sambarino-indicator,sambarino-orbital} established analogues of  our counting 
and equidistribution results in the setting of traditional ``uncusped'' Anosov representations. In this section, we will sketch how
to establish (mild generalizations of) his results in our framework. We start by recalling a characterization of Anosov representations of word hyperbolic groups established by
Kapovich-Leeb-Porti \cite{KLP2} and Bochi-Potrie-Sambarino \cite{BPS}.

If $\Gamma$ is a word hyperbolic group, then a representation $\rho:\Gamma\to\mathsf{SL}(d,\mathbb R)$ is
$P_k$-Anosov if there exist $A,a>0$ so that
$$\frac{\sigma_k(\rho(\gamma))}{\sigma_{k+1}(\rho(\gamma))}\ge Ae^{a|\gamma|}$$
for all $\gamma\in\Gamma$, where $|\gamma|$ is the word length of $\gamma$ with respect to some
fixed generating set on $\Gamma$. In this case, it is known (see \cite{BCLS} or \cite{CLT}) that there is a finite Markov shift $(\Sigma_\Gamma^+,\sigma)$ for
the geodesic flow of $\Gamma$ and a surjective map
$$G:\bigcup_{n\in\mathbb N} \mathrm{Fix}^n\to[\Gamma].$$
Moreover, if $\theta\subset\{1,\ldots,d-1\}$ is non-empty and symmetric, $\rho$ is $\theta$-Anosov, and
$\phi\in\mathfrak{a}_\theta\cap\mathcal B(\rho)^+$, then
there exists a H\"older continuous function $\tau_\rho^\phi:\Sigma_\Gamma^+\to\mathbb R$ so
that if $x\in\mathrm{Fix}^n\subset\Sigma_\Gamma^+$, then
$$ S_n\tau_\rho^\phi(x)=\phi(\ell(\rho(G(x)))).$$

Lalley \cite[Theorems\ 5\ and\ 7]{lalley} established analogues of our counting and equidistribution results for finite Markov shifts.
Moreover, our proofs generalize his techniques so they go through in the setting of finite Markov shifts without any assumptions on entropy gap. 

\begin{cor}
Suppose that  $\Gamma$ is a word hyperbolic group, $\theta\subset\{1,\ldots,d-1\}$ is non-empty and symmetric, and
$\rho:\Gamma\to\mathsf{SL}(d,\mathbb R)$ is $\theta$-Anosov.
If $\phi\in\mathfrak{a}_\theta^*\cap\mathcal B(\rho)^+$, then
there exists a unique $\delta_\phi(\rho)>0$ so that $P(-\delta_\phi(\rho) \tau_\rho^\phi)=0$ and
$${\displaystyle \lim_{t\to\infty}M_{\phi}(t)\frac{t\delta_\phi(\rho)}{e^{t\delta_\phi(\rho)}}}=1$$
where 
$$M_{\phi}(t)=\#\Big\{[\gamma]\in[\Gamma]\ \big|\ \phi(\ell(\rho(\gamma)))\le t\Big\}.$$
\end{cor}

\begin{proof}
Our proof of property (6) in Theorem \ref{Roof Properties}* gives immediately that $\tau_\rho^\phi$ is
non-arithmetic, which is the only assumption needed to apply our Theorem \ref{thm:countingN} or Theorem 7 in \cite{lalley}
in the setting of a finite Markov shift.
\end{proof}

We also obtain a Manhattan Curve theorem, which does not seem to have appeared  in print before in this generality,
but was certainly well-known to experts. In particular, Sambarino \cite[Proposition 4.7]{sambarino-indicator} describes
a closely related phenomenon for Borel Anosov representations.

\begin{cor}
Suppose that $\Gamma$ is a word hyperbolic group, $\theta\subset\{1,\ldots,d-1\}$ is non-empty and symmetric,  and that 
$\rho:\Gamma\to\mathsf{SL}(d,\mathbb R)$ and
$\eta:\Gamma\to\mathsf{SL}(d,\mathbb R)$ are $\theta$-Anosov.
If $\phi\in\mathfrak{a}_\theta^*\cap\mathcal{B}(\rho)^+\cap\mathcal{B}(\eta)^+$ and
$$\mathcal C^{\phi}(\rho,\eta)=\big\{ (a,b)\in \mathbb R^2 | \ a+b>0\   \mathrm{and}\  P(-a\tau_\rho^{\phi}-b\tau_\eta^{\phi})=0\big\},$$
then 
\begin{enumerate}
\item
$\mathcal C^\phi(\rho,\eta)$ is an analytic curve,
\item
$(\delta_{\phi}(\rho),0)$ and $(0,\delta_{\phi}(\eta))$ lie on $\mathcal C^\phi(\rho,\eta)$,
\item
and $\mathcal C^\phi(\rho,\eta)$ is strictly convex, unless 
$$\ell^{\phi}(\rho(\gamma))=\frac{\delta_\phi(\eta)}{\delta_\phi(\rho)}\ell^{\phi}(\eta(\gamma))$$ 
for all 
$\gamma\in\Gamma$.
\end{enumerate}
Moreover,  the tangent line to $\mathcal C^{\phi}(\rho,\eta)$ at $(\delta_{\phi}(\rho),0)$ has slope
$$-I^{\phi}(\rho,\eta)=-\frac{\int \tau_\eta^{\phi} dm_{-\delta_{\phi}(\rho)\tau^{\phi}_\rho}}{\int \tau_\rho^{\phi}\ 
dm_{-\delta_{\phi}(\rho)\tau^{\phi}_\rho}} 
$$
\end{cor}

The analogues of Corollaries \ref{intersection rigidity} and \ref{geometric intersection} appear in \cite[Section 8]{BCLS} as consequences 
of classical  Thermodynamical results of Bowen, Pollicott and Ruelle \cite{bowen, bowen-ruelle,pollicott,ruelle}.

\medskip\noindent
{\bf Historical Remarks:} In the counting estimates and equistribution results  in his papers, Sambarino
assumes that $\rho$ is irreducible if $\theta=\{1,d-1\}$  (see \cite{sambarino-quantitative}) or Zariski dense if $\rho$ is Borel Anosov
(see \cite{sambarino-indicator,sambarino-orbital})
and that $\Gamma=\pi_1(M)$ where $M$ is a negatively curved manifold. However, after \cite{BCLS} the generalizations
stated here would certainly have been well-known to him.
Carvajales \cite[Appendix A]{carvajales-quadratic} uses results from \cite{BCLS} to explain how one can
remove the assumption that $\Gamma=\pi_1(M)$ in Sambarino's work. The removal of the irreducibility 
assumption follows from the construction of the semi-simplification in \cite{GGKW}. Pollicott and Sharp \cite{PS-Hitchin} independently
derived related counting results for Hitchin representations.

\end{document}